\documentclass[a4paper,leqno,12pt]{amsart}
\usepackage[all, 2cell, knot,]{xy} \UseAllTwocells \SilentMatrices
\xyoption{rotate}
\usepackage{xspace}
\usepackage{amsmath}
\usepackage{amstext}
\usepackage{amsfonts}
\usepackage[mathscr]{euscript}
\usepackage{amscd}
\usepackage{latexsym}
\usepackage{amssymb}
\usepackage{layout}
\usepackage{amsthm}
\usepackage{amsfonts}
\usepackage{amsxtra}
\usepackage{color}
\usepackage[usenames,dvipsnames,svgnames,table]{xcolor}
\usepackage{graphics}
\usepackage{bm, amsbsy}
\usepackage{tabulary}
\usepackage{etoolbox}
\usepackage{fancyhdr}
\usepackage[bookmarks=true,hyperfootnotes=false]{hyperref}
\usepackage{tikz-cd}
\usepackage{tikz}
\usepackage{bbm}
\usepackage{xargs}
\usepackage{standalone}
\usepackage{relsize}
\usepackage{chngcntr}
\usepackage[strict]{changepage}
\usepackage{amsrefs}
\usepackage{marginnote}
\hypersetup{
			linktoc=all,
            colorlinks=true,
			linkcolor=blue,
			anchorcolor=black,
			citecolor=green,
			urlcolor=black,
}
\begin{document}
%
%
%

\theoremstyle{plain}
\swapnumbers
    \newtheorem{theorem}[figure]{Theorem}
    \newtheorem{proposition}[figure]{Proposition}
    \newtheorem{lemma}[figure]{Lemma}
    \newtheorem{claim}[figure]{Claim}
    \newtheorem{keylemma}[figure]{Key Lemma}
    \newtheorem{corollary}[figure]{Corollary}
    \newtheorem{fact}[figure]{Fact}
    \newtheorem{subsec}[figure]{}
    \newtheorem*{thmb}{Theorem B}
    \newtheorem*{thma}{Theorem A}
    \newtheorem*{thmc}{Theorem C}
    \newtheorem*{thmd}{Theorem D}
\theoremstyle{definition}
    \newtheorem{definition}[figure]{Definition}
    \newtheorem{notation}[figure]{Notation}
     \newtheorem{question}[figure]{Question}
\theoremstyle{remark}
    \newtheorem{remark}[figure]{Remark}
    \newtheorem{example}[figure]{Example}
    \newtheorem{ack}[figure]{Acknowledgements}
\renewcommand{\thefigure}{\arabic{section}.\arabic{figure}}
%
%
%
\newenvironment{myeq}[1][]
{\stepcounter{figure}\begin{equation}\tag{\thefigure}{#1}}
{\end{equation}}
\newcommand{\myeqn}[2][]
{\stepcounter{figure}\begin{equation}
     \tag{\thefigure}{#1}\vcenter{#2}\end{equation}}
\newcommand{\mydiag}[2][]{\myeq[#1]\xymatrix{#2}}
\newcommand{\mydiagram}[2][]
{\stepcounter{figure}\begin{equation}
     \tag{\thefigure}{#1}\vcenter{\xymatrix{#2}}\end{equation}}
\newcommand{\myodiag}[2][]
{\stepcounter{figure}\begin{equation}
     \tag{\thefigure}{#1}\vcenter{\xymatrix@C=1pc@R=.5pc{#2}}\end{equation}}
\newcommand{\mypdiag}[2][]
{\stepcounter{figure}\begin{equation}
     \tag{\thefigure}{#1}\vcenter{\xymatrix@C=4.5pc@R=1.8pc{#2}}\end{equation}}
\newcommand{\myqdiag}[2][]
{\stepcounter{figure}\begin{equation}
     \tag{\thefigure}{#1}\vcenter{\xymatrix@C=30pt{#2}}\end{equation}}
\newcommand{\myrdiag}[2][]
{\stepcounter{figure}\begin{equation}
     \tag{\thefigure}{#1}\vcenter{\xymatrix@C=50pt@L=2pt{#2}}\end{equation}}
\newenvironment{mysubsection}[2][]
{\begin{subsec}\begin{upshape}\begin{bfseries}{#2.}
\end{bfseries}{#1}}
{\end{upshape}\end{subsec}}
\newenvironment{mysubsect}[2][]
{\begin{subsec}\begin{upshape}\begin{bfseries}{#2\vsn.}
\end{bfseries}{#1}}
{\end{upshape}\end{subsec}}
\newcommand{\sect}{\setcounter{figure}{0}\section}
%
%
%
\newcommand{\wh}{\ -- \ }
\newcommand{\wwh}{-- \ }
\newcommand{\w}[2][ ]{\ \ensuremath{#2}{#1}\ }
\newcommand{\ww}[1]{\ \ensuremath{#1}}
\newcommand{\www}[2][ ]{\ensuremath{#2}{#1}\ }
\newcommand{\wwb}[1]{\ \ensuremath{(#1)}-}
\newcommand{\wb}[2][ ]{\ (\ensuremath{#2}){#1}\ }
\newcommand{\wref}[2][ ]{\ (\ref{#2}){#1}\ }
\newcommand{\wwref}[3][ ]{\ (\ref{#2})-(\ref{#3}){#1}\ }
%
%
\newcommand{\hs}{\hspace*{5 mm}}
\newcommand{\hsi}{\hspace*{-16 mm}}
\newcommand{\hsl}{\hspace*{-10 mm}}
\newcommand{\hslm}{\hspace*{-38 mm}}
\newcommand{\hsln}{\hspace*{-29 mm}}
\newcommand{\hslo}{\hspace*{-20 mm}}
\newcommand{\hslp}{\hspace*{-5 mm}}
\newcommand{\hsn}{\hspace*{1 mm}}
\newcommand{\hsm}{\hspace*{2 mm}}
\newcommand{\hsp}{\hspace*{5 mm}}
\newcommand{\hsq}{\hspace*{23 mm}}
\newcommand{\vsn}{\vspace{2 mm}}
\newcommand{\vs}{\vspace{5 mm}}
\newcommand{\vsm}{\vspace{3 mm}}
%
%
\newcommand{\blb}[1]{\mathbb{#1}}
\newcommand{\bbA}{\mathbb{A}}
\newcommand{\bbB}{\mathbb{B}}
\newcommand{\bbC}{\mathbb{C}}
\newcommand{\bbD}{\mathbb{D}}
\newcommand{\bbE}{\mathbb{E}}
\newcommand{\bbF}{\mathbb{F}}
\newcommand{\bbG}{\mathbb{G}}
\newcommand{\bbH}{\mathbb{H}}
\newcommand{\bbN}{\mathbb{N}}
\newcommand{\bbP}{\mathbb{P}}
\newcommand{\bbR}{\mathbb{R}}
\newcommand{\bbS}{\mathbb{S}}
\newcommand{\bbX}{\mathbb{X}}
\newcommand{\bbY}{\mathbb{Y}}
\newcommand{\bbZ}{\mathbb{Z}}
%
%
\newcommand{\bol}[1]{\bm{#1}}
\newcommand{\boA}{\bm{A}}
\newcommand{\boB}{\bm{B}}
\newcommand{\boC}{\bm{C}}
\newcommand{\boD}{\bm{D}}
\newcommand{\boE}{\bm{E}}
\newcommand{\boF}{\bm{F}}
\newcommand{\boG}{\bm{G}}
\newcommand{\boH}{\bm{H}}
\newcommand{\boN}{\bm{N}}
\newcommand{\boP}{\bm{P}}
\newcommand{\boR}{\bm{R}}
\newcommand{\boS}{\bm{S}}
\newcommand{\boX}{\bm{X}}
\newcommand{\boY}{\bm{Y}}
\newcommand{\boZ}{\bm{Z}}
%
%
\newcommand{\clA}{\mathcal{A}}
\newcommand{\clB}{\mathcal{B}}
\newcommand{\clC}{\mathcal{C}}
\newcommand{\clD}{\mathcal{D}}
\newcommand{\clE}{\mathcal{E}}
\newcommand{\clF}{\mathcal{F}}
\newcommand{\Fs}{\clF\sb{\bullet}}
\newcommand{\clG}{\mathcal{G}}
\newcommand{\clH}{\mathcal{H}}
\newcommand{\clI}{\mathcal{I}}
\newcommand{\clK}{\mathcal{K}}
\newcommand{\ck}{\mathcal{K}}
\newcommand{\clL}{\mathcal{L}}
\newcommand{\clM}{\mathcal{M}}
\newcommand{\clN}{\mathcal{N}}
\newcommand{\clO}{\mathcal{O}}
\newcommand{\clP}{\mathcal{P}}
\newcommand{\clR}{\mathcal{R}}
\newcommand{\clS}{\mathcal{S}}
\newcommand{\clT}{\mathcal{T}}
\newcommand{\ct}{\mathcal{T}}
\newcommand{\clU}{\mathcal{U}}
\newcommand{\clV}{\mathcal{V}}
\newcommand{\clW}{\mathcal{W}}
\newcommand{\clX}{\mathcal{X}}
\newcommand{\clY}{\mathcal{Y}}
\newcommand{\clZ}{\mathcal{Z}}

\newcommand{\za}{\alpha}
\newcommand{\zb}{\beta}
\newcommand{\zg}{\gamma}
\newcommand{\zG}{\Gamma}
\newcommand{\zd}{\delta}
\newcommand{\zD}{\Delta}
\newcommand{\ze}{\epsilon}
\newcommand{\zve}{\varepsilon}
\newcommand{\zf}{\phi}
\newcommand{\zl}{\lambda}
\newcommand{\zL}{\Lambda}
\newcommand{\zk}{\kappa}
\newcommand{\zt}{\theta}
\newcommand{\zvt}{\vartheta}
\newcommand{\zT}{\Theta}
\newcommand{\zo}{\omega}
\newcommand{\zO}{\Omega}
\newcommand{\zp}{\phi}
\newcommand{\zvp}{\varphi}
\newcommand{\zP}{\Phi}
\newcommand{\zr}{\rho}
\newcommand{\zs}{\sigma}
\newcommand{\zS}{\Sigma}
\newcommand{\pt}{\partial}
\newcommand{\zgu}[1]{\zg\sp{(#1)}}
\newcommand{\zgb}[1]{\zg\sb{(#1)}}
\newcommand{\elll}[1]{\ell\sb{#1}}
%
%
\newcommand{\ab}{\operatorname{ab}}
\newcommand{\catn}[1]{\operatorname{cat}\sp{#1}}
\newcommand{\Coeq}{\operatorname{Coeq}}
\newcommand{\Dec}{\operatorname{Dec}}
\newcommand{\Ner}{\operatorname{Ner}}
\newcommand{\Alg}{\operatorname{Alg}}
\newcommand{\AQ}{\operatorname{AQ}}
\newcommand{\Hom}{\operatorname{Hom}}
\newcommand{\Tot}{\operatorname{Tot}}
\newcommand{\Map}{\operatorname{Map}}
\newcommand{\map}{\operatorname{map}}
\newcommand{\hd}{\operatorname{hd}}
\newcommand{\Ho}{\operatorname{ho}}
\newcommand{\Id}{\operatorname{Id}}
\newcommand{\inc}{\operatorname{inc}}
\newcommand{\Obj}{\operatorname{Obj}\,}
\newcommand{\op}{\sp{\operatorname{op}}}
\newcommand{\ord}{\operatorname{Or}}
\newcommand{\wg}{\operatorname{wg}}
\newcommand{\Hn}{\clH\sb{n}}
\newcommand{\Disc}{\operatorname{Disc}}
\newcommand{\Diz}{\Disc\sb{0}}
\newcommand{\pr}{\operatorname{pr}}
\newcommand{\Diag}{\operatorname{Diag}}
\newcommand{\diag}{\operatorname{diag}}
\newcommand{\bsim}{/\!\!\sim}
%
%
\newcommand{\nid}{\noindent}
\newcommand{\ds}{\displaystyle}
\newcommand{\bk}{\bigskip}
\newcommand{\mk}{\medskip}
\newcommand{\smk}{\smallskip}
\newcommand{\ovl}[1]{\overline{#1}}
\newcommand{\ovll}[1]{\overset{=}{#1}}
\newcommand{\up}[1]{\sp{(#1)}}
\newcommand{\lo}[1]{\sb{(#1)}}
\newcommand{\rw}{\rightarrow}
\newcommand{\Rw}{\Rightarrow}
\newcommand{\lw}{\leftarrow}
\newcommand{\Lw}{\Leftarrow}
\newcommand{\xrw}{\xrightarrow} 
\newcommand{\xlw}{\xleftarrow} 
\newcommand{\hxrw}[1]{\xymatrix{\ \ar@{\sp{(}->}\sp{#1}[r] & \ }}
\newcommand{\tiund}[1]{{\times}\sb{#1}\:}
\newcommand{\pro}[3]{#1\tiund{#2}\overset{#3}{\cdots}\tiund{#2}#1}
\newcommand{\tens}[2]{#1\,\tiund{#2}\,#1}
\newcommand{\uset}[2]{\underset{#1}{#2}}
\newcommand{\oset}[2]{\overset{#1}{#2}}
\newcommand{\mi}{\text{-}}
\newcommand{\nm}{(n-1)}
\newcommand{\bl}{\bullet}
\newcommand{\cop}{\textstyle{\,\coprod\,}}
\newcommand{\seq}[3]{{#1}\sb{#2}...{#1}\sb{#3}}
\newcommand{\seqc}[3]{{#1}\sb{#2},...,{#1}\sb{#3}}
\newcommand{\ssr}{\!\!\!\!\!\!\!\!\!\!\!\!\!\!}
\newcommand{\cirsm}{\scriptstyle{\circ}\textstyle}
\newlength{\myline}
\setlength{\myline}{0.7pt}
%
%
\newcommand{\dop}[1]{\Delta\sp{{#1}\op}}
\newcommand{\Dop}{\Delta\op}
\newcommand{\Dnop}{\Delta\sp{n\op}}
\newcommand{\Dmenop}{\Delta\sp{{n-1}\op}}
\newcommand{\cat}[1]{\mbox{$\mathsf{Cat\sp{#1}}$}}
\newcommand{\Cat}{\mbox{$\mathsf{Cat}$}}
\newcommand{\cathd}[1]{\Cat\sb{\hd}\sp{#1}}
\newcommand{\Cato}{\mbox{$\mathsf{Cat\sb{\clO}}$}}
\newcommand{\catwg}[1]{\Cat\sb{wg}\sp{#1}}
\newcommand{\Catwg}[1]{\Cat\sb{wg}\sp{#1}}
\newcommand{\GCat}{\mbox{$\mathsf{GCat}$}}
\newcommand{\Gcatwg}[1]{\GCat\sb{\wg}\sp{#1}}
\newcommand{\gcatwg}[1]{\GCat\sb{\wg}\sp{#1}}
\newcommand{\Gpd}{\mbox{$\mathsf{Gpd}$}}
\newcommand{\Gpdn}[1]{\Gpd\sp{#1}}
\newcommand{\Gpdo}{\Gpd\sb{\clO}}
\newcommand{\GC}{\Gpd\,\clC}
\newcommand{\GnC}[1]{\Gpd\sp{#1}\clC}
\newcommand{\Gph}{\mbox{$\mathsf{Gph}$}}
\newcommand{\Gpho}{\Gph\sb{\clO}}
\newcommand{\GGpdo}{2\mi\Gpd\sb{\clO}}
\newcommand{\Gpdwg}[1]{\Gpd\sp{#1}\sb{\wg}}
\newcommand{\Gpdhd}[1]{\Gpd\sp{#1}\sb{\hd}}
\newcommand{\Spl}{\mbox{$\mathsf{Spl}$}}
\newcommand{\Top}{\mbox{$\mathsf{Top}$}}
\newcommand{\dI}{D}
\newcommand{\odI}{D}
\newcommand{\Ne}[1]{N\sb{#1}}
\newcommand{\Nb}[1]{N\sb{(#1)}}
\newcommand{\Nu}[1]{N\sp{(#1)}}
\newcommand{\oN}[1]{N\sb{#1}}
\newcommand{\di}[1]{d\sp{(#1)}}
\newcommand{\dn}{d\sp{(n)}}
\newcommand{\bd}{\bar d}
\newcommand{\Dn}{D\sb{n}}
\newcommand{\Dnm}{D\sb{n-1}}
\newcommand{\tld}{\tilde{d}}
\newcommand{\D}[1]{D\sb{#1}}
\newcommand{\om}[1]{\bar{m}\up{#1}}
\newcommand{\p}[1]{p\up{#1}}
\newcommand{\bp}{\bar p}
\newcommand{\bap}[1]{\bar{p}\up{#1}}
\newcommand{\pn}{p\sp{(n)}}
\newcommand{\qu}[1]{q\sp{(#1)}}
\newcommand{\baq}[1]{\bar{q}\up{#1}}
\newcommand{\qn}{q\sp{(n)}}
\newcommand{\bam}[1]{\bar{m}\up{#1}}
\newcommand{\Tan}{\mbox{$\mathsf{Ta\sp{n}}$}}
\newcommand{\ta}[1]{\mbox{$\mathsf{Ta\sp{#1}}$}}
\newcommand{\gta}[1]{\mbox{$\mathsf{GTa\sp{#1}}$}}
\newcommand{\Trt}{\mbox{$\mathsf{Track\sb{\clO}}$}}
\newcommand{\Trk}[1]{\mbox{$\mathsf{Track\sp{#1}\sb{\clO}}$}}
\newcommand{\Set}{\mbox{$\mathsf{Set}$}}
\newcommand{\St}{St\,}
\newcommand{\st}[1]{St\sp{(#1)}}
\newcommand{\Tr}{Tr\,}
\newcommand{\tr}[1]{Tr\sb{#1}}
\newcommand{\uh}{\underline{h}}
\newcommand{\uk}{\underline{k}}
\newcommand{\uv}{\underline{v}}
\newcommand{\ur}{\underline{r}}
\newcommand{\us}{\underline{s}}
\newcommand{\iov}{\ovl{I}}
\newcommand{\ion}{\ovl{N}}
\newcommand{\ovdia}{\diag\,}
\newcommand{\Ovdia}{\Diag\,}
\newcommand{\rz}{R\sb{0}}
\newcommand{\Qn}{Q\sb{n}}
\newcommand{\Qnm}{Q\sb{n-1}}
\newcommand{\Discn}{Disc\sb{n}}
\newcommand{\Discz}{\Disc\sb{0}}
\newcommand{\funcat}[2]{[\Delta\sp{{#1}\sp{\op}},#2]}
\newcommand{\ps}{\sf{ps}}
\newcommand{\LL}{\mbox{\sf{L}}}
\newcommand{\Lb}[1]{\mbox{$\mathsf{L\sb{(#1)}}$}}
\newcommand{\nfol}{$n$-fold }
\newcommand{\Tm}{\mbox{\sf GTa}\sb{2}}
\newcommand{\Ta}{\mbox{\sf Ta}\sb{2}}
\newcommand{\tms}[2]{\overset{#1}{\times}\sb{#2}}
\newcommand{\btil}{\widetilde{B}}
\newcommand{\orn}[1]{\ord\sb{(#1)}}
\newcommand{\tiq}{\oset{q}{\times}}
\newcommand{\nty}{\mbox{$n$-types}}
\newcommand{\gpd}[1]{\mbox{$\mathsf{Gpd}\sp{#1}$}}
\newcommand{\gpdwg}[1]{\mbox{$\mathsf{Gpd\sb{\wg}\sp{#1}}$}}
\newcommand{\Rbt}[1]{\clR\sb{#1}}
\newcommand{\Lbt}[1]{\clL\sb{#1}}
\newcommand{\Pz}{\Pi\sb{0}}
\newcommand{\nab}{\nabla}
\newcommand{\sbl}{\scriptstyle{\bullet}\textstyle}
\newcommand{\ckb}{\ck\sb{\sbl}}
\newcommand{\HAlg}[1]{H\sb{\Alg}\sp{#1}}
\newcommand{\Sa}{\clS\sb{\ast}}
\newcommand{\hy}[2]{{#1}\text{-}{#2}}
\newcommand{\hpi}{\hat{\pi}\sb{1}}
\newcommand{\iO}[1]{({#1},\!\clO)}
\newcommand{\hC}[1]{\hy{#1}{\Cat}}
\newcommand{\hCo}[1]{\hy{#1}{\Cato}}
\newcommand{\hpC}[1]{\hy{({#1})}{\Cat}}
\newcommand{\hpCo}[1]{\hy{({#1})}{\Cato}}
\newcommand{\iOC}[1]{\hC{\iO{#1}}}
\newcommand{\SO}{\iO{\clS}}
\newcommand{\HAQ}[1]{H\sb{\AQ}\sp{#1}}
\newcommand{\HAQB}[1]{H\sb{\AQ/B}\sp{#1}}
\newcommand{\OC}{\hy{\clO}{\Cat}}
\newcommand{\SOC}{\iOC{\clS}}
\newcommand{\coH}[3]{H\sp{#1}({#2};{#3})}
%
%
\newcommand{\Gd}{G\sb{\bullet}}
\newcommand{\mX}{\mathscr{X}}
\newcommand{\mY}{\mathscr{Y}}
\newcommand{\mZ}{\mathscr{Z}}
\newcommand{\dd}{\operatorname{d}}
\newcommand{\dpp}[1]{\dd\up{#1}}
\newcommand{\dZ}{\dd\!Z}
%
%
\newcommand{\co}[1]{c({#1})}
\newcommand{\csk}[1]{\operatorname{csk}\sb{#1}}
\newcommand{\Po}[1]{P\sp{{#1}}}
\newcommand{\sk}[1]{\operatorname{sk}\sb{#1}}
\newcommand{\EM}[3]{E\sp{#1}({#2},{#3})}
\newcommand{\EB}[2]{E\sb{B}({#1},{#2})}
\newcommand{\EL}[2]{E\sb{\hpi\mX}({#1},{#2})}
\newcommand{\EX}[2]{E\sb{X}({#1},{#2})}
\newcommand{\EQM}[1]{E\sb{\clC}\up{#1}(Q,M)}
%
\newcommand{\Sc}[1]{\scriptstyle{#1}}
\newcommand{\Scc}[1]{\scriptscriptstyle{#1}}
\newcommand{\Sz}[1]{\scriptsize{#1}}

\newcommand{\arsx}[1]{\ar@<1.5ex>@/\sb{0}.5pc/@{-}[#1]\ar@<1.5ex>@/\sb{0}.55pc/@{-}[#1]} 
\newcommand{\ardx}[1]{\ar@<-2.5ex>@/\sp{0}.5pc/@{-}[#1]\ar@<-2.5ex>@/\sp{0}.55pc/@{-}[#1]}

\newcommand{\xdownarrow}[1]{%
  {\left\downarrow\vbox to #1{}\right.\kern-\nulldelimiterspace}%
}
\newcommand{\DV}{\oset{\Large\vdots}{\big\downarrow}}

\setcounter{tocdepth}{2}
\DeclareRobustCommand{\VAN}[3]{#2} 
\title{A Model for the Andr\'{e}-Quillen Cohomology of an \wwb{\infty,1}Category}

%

%
\author{David Blanc}
\address{Department of Mathematics\\ University of Haifa\\ 3498838 Haifa\\ Israel}
\email{blanc@math.haifa.ac.il}
\author{Simona Paoli}
\address{Department of Mathematics - School of Computing and Natural Sciences\\
  University of Aberdeen\\ Aberdeen, AB24 3FX, UK}
\email{simona.paoli@abdn.ac.uk}
\date{\today}
\begin{abstract}
We describe a comonad on $n$-track categories, for each \w[,]{n\geq 0}
yielding an explicit cosimplicial abelian group model for the Andr\'{e}-Quillen
cohomology of an \wwb{\infty,1}category.
\end{abstract}
\maketitle
\setcounter{section}{-1}

%
%
\sect{Introduction}

A first approximation to understanding topological spaces is provided by their
homotopy category \w[,]{\Ho\Top} where continuous maps are replaced by their homotopy
classes \wh that is, by the set of connected components of the space
\w{\Map(X,Y)} of maps between two topological spaces $X$ and $Y$. However, there is higher
order homotopy information captured only by the full homotopy type of \w{\Map(X,Y)}
itself.  This leads to the study of the category \w{\Top} equipped with its
enrichment in topological spaces.

In order to make this enrichment more approachable, it is useful to have intermediate
levels of complexity between \w{\Ho\Top} and \w[.]{\Top} One natural choice is
to consider successive truncations of each mapping space \w{M=\Map(X,Y)} provided
by its \emph{Postnikov system}: recall that for any topological space $M$, we have
a tower of fibrations
\begin{myeq}[\label{eqpostsystsp}]
M~\to~\dotsc~\to~\Po{n}M~\to~\Po{n-1}M~\to~\dotsc~\to~\Po{1}M~\to~\Po{0}M~,
\end{myeq}
\noindent where \w{\pi\sb{i}\Po{n}M=\pi\sb{i}M} if \w[,]{i\leq n} and zero otherwise:
that is, \w{\Po{n}M} is an $n$-\emph{type}. In particular
\w{\Po{0}\Map(X,Y)\simeq[X,Y]} (a discrete topological space).

It turns out to be more convenient to work with \emph{simplicial categories} \wh that is,
those enriched in simplicial sets, rather than topological spaces (both carry the
same homotopy information). This is because for a Kan complex $M$, \w{\Po{n}M} is
just the \wwb{n+1}coskeleton \w[,]{\csk{n+1}M} and the functor \w{\csk{n+1}} is
strictly monoidal. Thus if $\mX$ is a fibrant simplicial category (i.e., one enriched
in Kan complexes), we have a tower of simplicial categories
\begin{myeq}[\label{eqpostsystcat}]
  \mX~\to~\dotsc~\to~\Po{n}\mX~\to~\Po{n-1}\mX~\to~\dotsc~\to~\Po{1}\mX~\to~\Po{0}\mX~
  \simeq~\Ho\mX~,
\end{myeq}
\noindent obtained by applying \wref{eqpostsystsp} to each mapping space of $\mX$,
where for \w[,]{n\geq 0} \w{\Po{n}\mX} is a simplicial category enriched in $n$-types.

For a topological space (or simplicial set) $M$, \w{\Po{n}M} is obtained from
\w{\Po{n-1}M} as the homotopy fiber of a certain map
\w[,]{k\sb{n-1}:\Po{n-1}M\to E\sb{\hpi M}(\pi\sb{n}M, n+1)} called the \wwb{n-1}st
$k$-\emph{invariant} for $M$, landing in a (twisted) Eilenberg Mac~Lane space representing
the cohomology group \w{H\sp{n+1}(\Po{n-1}M,\pi\sb{n}M)} (with local coefficients).

As Dwyer and Kan showed in \cite{DKanO}, this also holds for any simplicial category
$\mX$, where for any \ww{\hpi\mX}-module $\clD$, such as \w[,]{\clD=\pi\sb{n}\mX}
we have an Eilenberg-Mac~Lane simplicial category \w[,]{\EL{\clD}{n}} representing
what they called the \ww{\SO}-\emph{cohomology} of the simplicial category $\mX$.
Again \w{\Po{n}\mX} is obtained from \w{\Po{n-1}\mX} as the homotopy fiber of a certain
map \w{k\sb{n-1}:\Po{n-1}\mX\to\EL{\pi\sb{n}M}{n+1}} of simplicial categories.
Thus ultimately, all homotopy information about a simplicial category $\mX$ is
encoded inductively by the sequence of ``modules''
\w[,]{\{\pi\sb{n}\mX\}\sb{n=0}\sp{\infty}} together with the $k$-invariants of $\mX$
(thought of as cohomology classes).  Moreover, these $k$-invariants can be used to
extract various higher homotopy invariants of $\mX$ (see \cite{BMeadA}), and in particular
yield an obstruction theory for realizing homotopy-commutative diagram (see \cite{DKanO}).

This version of cohomolgy was later re-interpreted and extended by Harpaz, Nuiten,
and Prasma, who defined the (spectral) \emph{Andr\'{e}-Quillen cohomology} of
various types of (enriched) $\infty$-categories (see \cite{HNPrasA,HNPrasT}).
This generalized also the original Andr\'{e}-Quillen cohomology first defined for algebras
by Andr\'{e} (in \cite{AndrM}) and Quillen (in \cite{QuilH,QuilC}).
In this paper we restrict attention to simplicial categories, which appear to be
the most convenient version of \wwb{\infty,1}categories for our purposes
(see \cite{LurieHTT,BergI}).

Note that \w{\Po{n}} is a homotopical localization functor, and the
$n$-th Eilenberg-Mac~Lane simplicial category \w{\EL{\clD}{n}} is its own $n$-th
Postnikov section. Therefore, there is a canonical equivalence
\begin{myeq}[\label{eqpostnikov}]
\HAQ{n}(\mX;\clD)~:=~[\mX,\ \EL{\clD}{n}]~\cong~[\Po{n}\mX,\ \EL{\clD}{n}]~.
\end{myeq}
\noindent We conclude that in order to study the $n$-th Andr\'{e}-Quillen cohomology
group of a simplicial category, it suffices to look at simplicial categories which
are $n$-truncated \wh in other words enriched in $n$-types.

This has the advantage that we
may use one of the algebraic models for $n$-types to produce an algebraic
replacement for \w[.]{\Po{n}\mX}
The archtypical case is the fundamental groupoid \w{\hpi X} of a
(not necessarily connected) space $X$, which completely models the homotopy type of
\w[.]{\Po{1}X}  More generally, we are looking for a sequence of ``algebraic'' categories
\w[,]{(\clC\sp{n})\sb{n=1}\sp{\infty}} each equipped with a suitable notion of weak
equivalence $\simeq$, such that \w{\clC\sp{n}/\simeq} is equivalent to the homotopy
category of $n$-types of topological spaces.

In the path-connected case, these models include the \ww{\catn{n}}-groups of
\cite{LodaS}, the crossed $n$-cubes of \cite{ESteH} and \cite{TPorN},
the $n$-hyper-crossed complexes of \cite{CCegG}, and the weakly globular
\ww{\catn{n}}-groups of \cite{PaolW}.
However, for our purposes it is essential to have a model for non-connected $n$-types,
since the mapping spaces of a simplicial category are generally not path-connected.
Moreover, this model should have a convenient monoidal structure (preferably cartesian).

Thus we choose to work in the category \w{\Gcatwg{n}} of \emph{groupoidal weakly globular $n$-fold categories}, which is an algebraic model for $n$-types satisfying the required properties (see \cite{PaolS}). 
Furthermore, every $n$-type can be modeled by a \emph{weakly globular $n$-fold groupoid}, that is, an object in the full subcategory \w{\Gpdwg{n}} of \w{\Gcatwg{n}} (see \cite{BP}), which is more convenient algebraically. Thus our model for the Postnikov section \w{\Po{n}\mX} of a simplicial category $\mX$ is an \emph{$n$-track category}: that is, one enriched in \w[.]{\Gpdwg{n}} We use this to prove our main result:

\begin{thma}
There is a comonad $\ck$ on the category \w{\Trk{n}} of $n$-track categories
such that, for any $n$-track category $X$ and module $M$, the 
cosimplicial abelian group \w[,]{\Hom(\ckb,M)} obtained by mapping the simplicial
resolution \w{\ckb X} to $M$, computes the Andr\'{e}-Quillen cohomology of
$X$ with coefficients in $M$.
\end{thma}
\noindent See Corollary \ref{cor1-longseq} below.

\begin{mysubsection}{Notations and conventions}
\label{snac}
We denote by \w{\Cat} the category of small categories, and by \w{\Set} that of sets.
Let $\Delta$ denote the category of non-empty finite ordered sets and
order-preserving maps. A \emph{simplicial object} \w{\Gd} in a category $\clC$ is
a functor \w[,]{\Dop\to\clC} and the category of simplicial sets will be
denoted by \w[.]{\clS:=\Set\sp{\Dop}} A (small) \emph{simplicial category} may be
thought of as either a simplicial object in \w{\Cat} (with all structure functors
constant on objects), or as a  category enriched in $\clS$.

Given a small category $I$ and a functor $F:\clC\to\clD$, the functor
\begin{equation*}
\overline{F}:[I,\clC]\rw [I,\clD]
\end{equation*}
is obtained by applying $F$ levelwise; that is, for each $X\in [I,\clC]$
\begin{equation*}
  (\overline{F}X)(i)=F(X(i))\;.
\end{equation*}
By slight abuse of notation, in this paper we shall denote $\overline{F}$ simply by $F$.
\end{mysubsection}

\begin{mysubsection}{Organization}
\label{sorg}
Section \ref{prelim} is devoted to background on higher category theory and models for
$n$-types. Section \ref{comores} sets up a comonad on $n$-track categories, yielding
our cochain complex. Section \ref{cmodule} is devoted to the modules needed as
coefficients for the cohomology of an $n$-track category, and Section
\ref{ccohtrk} shows that the Andr\'{e}-Quillen cohomology of an $n$-track category
may be computed using the cochain complex we construct.
\end{mysubsection}

%
%
\sect{Background}\label{prelim}
In this paper we use the algebraic model of $n$-types given by the category
\w{\Gcatwg{n}} of groupoidal weakly globular $n$-fold categories. This is a
subcategory of the category \w{\Catwg{n}} of weakly globular $n$-fold categories
which was introduced in \cite{PaolS} as a model of weak $n$-categories.
We refer to \cite{PaolS} for the details of the inductive definition of
\w{\Gcatwg{n}} and of $n$-equivalences in it. Here we describe some of main features
of this model, which will be needed later on.

If $\clC$ is a category with pullbacks, there are a well-known categories \w{\Cat\clC}
of internal categories and internal functors and $\Gpd\clC\subset\Cat\clC$ of internal groupoids and internal functors (see for instance \cite{Borceux}).
There is also a nerve functor
$$
N:\Cat\clC\to\funcat{}{\clC}\;.
$$
Starting from \w{\Cat} and iterating this construction yields to the category
\w{\cat{n}} of $n$-fold categories \wh that is,
\w{\cat{1}= \Cat} and \w{\cat{n}=\Cat(\cat{n-1})} for \w[.]{n\geq 2} Similarly we have
the category \w{\Gpdn{n}} of $n$-fold groupoids, where \w{\Gpd\sp{1}=\Gpd} and
\w[.]{\Gpd\sp{n}=\Gpd(\Gpd\sp{n-1})}
Iterating the nerve construction we obtain the multinerve functor
\w[.]{\Nb{n}:\cat{n}\to\funcat{}{\Set}}
This functor is fully faithful. Hence we shall identify \w{\cat{n}} with the
subcategory \w{\Nb{n}\cat{n}} of \w{\funcat{n}{\Set}} which is the essential image
of the functor \w[.]{\Nb{n}} An $n$-fold category $X$ is called \emph{discrete}
when \w{\Nb{n}X} is a constant functor.

The category \w{\Gcatwg{n}} is a full subcategory of \w[;]{\cat{n}} when \w[,]{n=1}
\w{\Gcatwg{1}=\Gpd} is the category of groupoids. The $1$-equivalences in
\w{\Gcatwg{1}} are the equivalences of categories.

We first need to introduce a special class of $n$-fold categories: namely,
homotopically discrete $n$-fold categories.

\begin{definition}\label{def01}
Let \w[.]{\cathd{0}=\Set} Suppose, inductively, we have defined the subcategory
\w{\cathd{n-1}\subset \Cat\sp{n-1}} of
\emph{homotopically discrete \wwb{n-1}fold categories}. We say that the $n$-fold
category \w{X\in \Cat\sp{n}} is \emph{homotopically discrete} if:
\begin{itemize}
\item [a)] $X$ is a levelwise equivalence relation that is for each
\w{(k\sb{1},\ldots,k\sb{n-1})\in\dop{n-1}} and
\w{X\sb{k\sb{1},\ldots,k\sb{n-1}}} is an equivalence relation in \w[.]{\Cat}
\item [b)] \w{\p{n-1}X} is in \w[,]{\cathd{n-1}} where
\w{(\p{n-1}X)\sb{k\sb{1},\ldots,k\sb{n-1}}:=p X\sb{k\sb{1},\ldots,k\sb{n-1}}}
and \w{p\colon\Cat\to\Set} the isomorphism classes of objects functor.
\end{itemize}
When \w[,]{n=1} we denote by \w{\cathd{1}=\cathd{}} the subcategory of \w{\Cat}
consisting of equivalence relations.
\end{definition}

\begin{definition}\label{def02}
Let \w[.]{X\in \cathd{n}} Denote by \w{\zgu{n-1}\sb{X}:X\rw \p{n-1}X} the morphism
given by
$$
(\zgu{n-1}\sb{X})\sb{s\sb{1}...s\sb{n-1}} :X\sb{s\sb{1}...s\sb{n-1}} \rw p X\sb{s\sb{1}
  \dotsc s\sb{n-1}}\;.
$$
\noindent Let
$$
X\sp{d}~:=~\p{0}\p{1}\dotsc\p{n-1}X
$$
where $\p{0}=p$ and let \w{\zg\lo{n}} denote the composite
$$
X\xrw{\zgu{n-1}}\p{n-1}X \xrw{\zgu{n-2}} \p{n-2}\p{n-1}X \rw \cdots
\xrw{\zgu{0}} X\sp{d}\;.
$$
We call \w{\zg\lo{n}} the \emph{discretization map} of $X$.
\end{definition}

\begin{definition}\label{def03}
Given \w[,]{X\in\cathd{n}} for each \w[,]{a,b\in X\sb{0}\sp{d}} denote by
\w{X(a,b)} the fiber at \w{(a,b)} of the map
$$
  X\sb{1} \xrw{(d\sb{0},d\sb{1})} X\sb{0}\times X\sb{0}
  \xrw{\zg\lo{n}\times\zg\lo{n}} X\sb{0}\sp{d}\times X\sb{0}\sp{d}\;.
$$
It can be shown (see \cite{PaolS}) that \w[.]{X(a,b)\in \cathd{n-1}}
This should be thought of as a \ww{\Hom\mi(n-1)}-category.
\end{definition}

\begin{definition}\label{def04}
Define inductively $n$-equivalences in \w[.]{\cathd{n}} For \w[,]{n=1} a $1$-equivalence
is an equivalence of categories. Suppose we have defined \ww{\nm}-equivalences
in \w[.]{\cathd{n-1}} Then a map \w{f:X\rw Y} in \w{\cathd{n}} is
an \emph{$n$-equivalence} if
\begin{itemize}
\item [a)] \w{f(a,b):X(a,b) \rw Y(fa,fb)} is an \wwb{n-1}equivalence for all
  \w[.]{a,b \in X\sb{0}\sp{d}}
\item [b)] \w{\p{n-1}f} is an \wwb{n-1}equivalence.
\end{itemize}
\end{definition}

The next proposition means that homotopically discrete $n$-fold categories are
higher categorical 'fattening' of sets.

\begin{proposition}\label{pro01}
For any \w[,]{X\in\cathd{n}} the maps \w{\zg\up{n-1}:X\rw \p{n-1}X} and
\w{\zg\lo{n}:X\rw X\sp{d}} are $n$-equivalences.
\end{proposition}

The category \w{\Gcatwg{n}\subset\cat{n}} has the following properties:
\begin{itemize}
\item [a)] There are embeddings
\w[.]{\Gcatwg{n}\hookrightarrow \funcat{}{\Gcatwg{n-1}}\hookrightarrow\funcat{n-1}{\Cat}}
\item [b)] There is an \emph{\wwb{n-1}truncation} functor
\w{\p{n-1}:\Gcatwg{n}\to\Gcatwg{n-1}} given by
\w{(\p{n-1}X)\sb{k\sb{1}\cdots k\sb{n-1}}=pX\sb{k\sb{1}\cdots k\sb{n-1}}}
for each \w{(k\sb{1}\cdots k\sb{n-1})\in\Dmenop} and \w[.]{X\in\Gcatwg{n}}
\item [c)] \w[.]{\cathd{n}\subset \Gcatwg{n}}
\item [d)] For each $X$ in \w[,]{\Gcatwg{n}} \w{X\sb{0}} $\in$ \w[.]{\cathd{n-1}}
Hence there is a \wwb{n-1}equivalence \w[,]{\zg:X\sb{0}\rw X\sb{0}\sp{d}}
where \w{X\sb{0}\sp{d}} is a discrete $n$-fold category.\mk
\item [e)] For each \w[,]{k\geq 2} the induced Segal maps
\w{\hat{\mu}\sb{k}:X\sb{k}\to\pro{X\sb{1}}{X\sb{0}\sp{d}}{k}}
are \wwb{n-1}equivalences in \w[,]{\catwg{n-1}} where the maps
\w{\hat{\mu}\sb{k}} arise from the commutativity of the diagram
$$
\xymatrix@C=20pt{
&&&&& X\sb{k} \ar[llld]\sb{\nu\sb{1}} \ar[ld]\sb{\nu\sb{2}} \ar[rrd]\sp{\nu\sb{k}} &&&& \\
&& X\sb{1} \ar[ld]\sb{\zg d\sb{1}} \ar[rd]\sp{\zg d\sb{0}} &&
X\sb{1} \ar[ld]\sb{\zg d\sb{1}} \ar[rd]\sp{\zg d\sb{0}} && \dotsc &
X\sb{1} \ar[ld]\sb{\zg d\sb{1}} \ar[rd]\sp{\zg d\sb{0}} & \\
&X\sb{0}\sp{d} && X\sb{0}\sp{d} && X\sb{0}\sp{d} &\dotsc X\sb{0}\sp{d} && X\sb{0}\sp{d}
}
$$
\item [f)] Given $X\in\Gcatwg{n}$ and $(a,b)\in X\sb{0}\sp{d}\times X\sb{0}\sp{d}$ let
  \w{X(a,b)\subset X\sb{1}} be the fiber at $(a,b)$ of the map
$$
  X\sb{1}\xrw{(d\sb{0},d\sb{1})} X\sb{0}\times X\sb{0} \xrw{\zg\times\zg}
  X\sp{d}\sb{0}\times X\sp{d}\sb{0}\;.
$$

A map \w{f:X\rw Y} in \w{\Gcatwg{n}} is an \emph{$n$-equivalence} if the following
conditions hold:
\begin{itemize}
\item [i)] For all \w[,]{a,b\in X\sb{0}\sp{d}} \w{f(a,b): X(a,b)\rw Y(fa,fb)}
is a \wwb{n-1}equivalence.
\item [ii)] \w{\p{n-1}f} is a \wwb{n-1}equivalence.
\end{itemize}
\end{itemize}

One of the main results of \cite{PaolS} is that \w{\Gcatwg{n}} is an algebraic model
of $n$-types: that is,

\begin{theorem}[\cite{PaolS} Theorem 12.3.11]\label{paols}
There are functors
$$
\xymatrix{
  \gcatwg{n} \ar@/_/[rr]\sb{B {\cirsm }\Discn} &&
  \nty \ar@/_/[ll]\sb{\Qn {\cirsm} \clT\sb{n}}
}
$$
that induce equivalences of categories \w[,]{\gcatwg{n}\bsim\sp{n}\;\simeq\;\Ho(\nty)}
where \w{\gcatwg{n}\bsim\sp{n}} is the localization of \w{\gcatwg{n}} with respect
to $n$-equivalences and \w{\Ho(\nty)} is the homotopy category of $n$-types.
\end{theorem}

In this paper we use the following
\begin{definition}\label{def05}
The category \w{\Gpdwg{n}} of weakly globular $n$-fold groupoids is the full subcategory
  of \w{\Gcatwg{n}} whose objects are in \w[.]{\Gpd\sp{n}}
\end{definition}

\begin{remark}\label{rem001}
The category of weakly globular $n$-fold groupoids was originally introduced in
\cite[Definition 3.17]{BP}. The definition given there is equivalent to
Definition \ref{def05}, since the geometric weak equivalences in \w{\Gpdwg{n}}
coincide with the $n$-equivalences (\cite[Remark 6.37]{BP}).
\end{remark}
Using the results of \cite{BP}, in Theorem \ref{paols} one can use
an alternative functor
\begin{equation*}
  \Hn:n\text{-types }\to\Gcatwg{n}~,
\end{equation*}
\noindent simpler than the functor \w{Q\sb{n}\circ\clT\sb{n}} of Theorem \ref{paols}.
Furthermore, \w{\Hn} factors through the category \w[.]{\Gpdwg{n}}
We showed in \cite[Theorem 4.32]{BP} that the functors
\begin{equation*}
\Hn:n\text{-types }\to\Gpdwg{n}\hsp\text{and}\hsp B:\Gpdwg{n}\rw n\text{-types }
\end{equation*}
induce functors
\begin{equation*}
  \Ho(n\text{-types})\leftrightarrows \Gpdwg{n}\bsim^n
\end{equation*}
with \w[.]{B\,\Hn\cong\Id}

Given \w{X\in\funcat{n}{\clC}} and \w[,]{1\leq r\leq n} we denote by
\w{X\sp{\{r\}}\in\funcat{n-1}{\clC}} the $n$-fold simplicial object
\begin{equation*}
X\up{r}\sb{\seq{k}{1}{n}}=
\begin{cases}
X\sb{k\sb{2} k_3\ldots k_r k\sb{1} k\sb{r+1}\ldots k\sb{n}},&\text{if}\quad 1\leq r < n\\
    X\sb{k\sb{2} k_3\ldots k\sb{n-1} k\sb{1}}. & \text{if}\quad r=n \;.
\end{cases}
\end{equation*}
The functor sending $X$ to \w{X\up{r}} can be thought of as
`bringing the $r$-th index of $X$ to the fore'. Note that \w[.]{X\up{1}=X}

\begin{remark}\label{rem01}
  For each \w[,]{X\in\funcat{n}{\clC}} \w[,]{(\seqc{k}{1}{n-1})\in\dop{n-1}} and
  \w[:]{s\in\dop{}}
\begin{equation*}
\begin{split}
  (X\sb{s}\up{r})\sb{k\sb{1}}(\seqc{k}{2}{n-1}) =& X\sb{s}\up{r}(\seqc{k}{1}{n-1})~=~
  X\sb{\seq{k}{1}{r}\,s\,\seq{k}{r+1}{n-1}}\\
  =&(X\sb{k\sb{1}})\sb{s}\up{r-1} (\seqc{k}{2}{n-1}).
\end{split}
\end{equation*}
Therefore, \w[.]{(X\sb{s}\up{r})\sb{k\sb{1}}=(X\sb{k\sb{1}})\sb{s}\up{r-1}}
\end{remark}

The following, which will be used in Section \ref{comores}, is a useful criterion for
an object of \w{\cat{n}} to be in \w[:]{\Gcatwg{n}}

\begin{proposition}\label{pro02}
If \w{X\in\cat{n}} satisfies:
\begin{itemize}
\item[(a)] \w{X\sb{k\sb{1} k\sb{2} \ldots k\sb{n-1}*}\in\Gpd} for all
  \w[;]{k\sb{1},k\sb{2},\ldots ,k\sb{n-1}}
\item[(b)] For all \w[,]{1\leq r < n} the \wwb{n-1}fold category
  \w{X\sb{0}\up{r}} is a levelwise equivalence relation;
\item[(c)] \w[.]{\p{n-1} X\in\Gcatwg{n-1}}
\end{itemize}
Then $X$ is in \w[.]{\Gcatwg{n}} Furthermore, if \w{X\in\Gpd\sp{n}} satisfies
(a) and (b) and \w[,]{\p{n-1}X\in\Gpdwg{n-1}} then $X$ is in \w[.]{\Gpdwg{n}}
\end{proposition}

\begin{proof}
By induction on $n$ (trivially true for \w[).]{n=1} Suppose the claim holds
for \w[.]{n-1} By \cite[Proposition 7.2.8]{PaolS}, $X$ is in \w[.]{\catwg{n}}
Since by hypothesis \w{\p{n-1}X} is in \w[,]{\Gcatwg{n-1}} by definition of
\w{\Gcatwg{n}} (see \cite[Definition 12.3.1]{PaolS}) it is enough to show that
\w{X\sb{s}} is in \w{\Gcatwg{n-1}} for each \w[.]{s\geq 0}
But this holds by the induction hypothesis since, by Remark \ref{rem01},
\w{(X\sb{s}\sp{\{r\}})\sb{0}=(X\sb{0})\sb{s}\sp{\{r-1\}}} is a levelwise
equivalence relation (as is \w[),]{X\sb{0}\sp{\{r\}}}
while \w{\p{n-2}X\sb{s}=\p{n-1}X\sb{s}} is in \w[,]{\Gcatwg{n-2}} since
by hypothesis \w{\p{n-1}X} is in \w[.]{\Gcatwg{n-1}}

If \w{X\in\Gpd\sp{n}} satisfies (a) and (b) and \w[,]{\p{n-1}X\in\Gpdwg{n-1}}
then by the above, $X$ is in \w[.]{\Gcatwg{n}}
Since it is also in \w[,]{\Gpd\sp{n}} by definition $X$ is in \w[.]{\Gpdwg{n}}
\end{proof}

\begin{proposition}\label{pro03}
Given \w[,]{X\in\Gcatwg{n}} let \w{r:Z\to\dpp{n-1}\p{n-1}X} be an
\wwb{n-1}equivalence in \w[.]{\Gcatwg{n-1}} Consider the pullback
\myqdiag{
P\ar[r]\sp{w}\ar[d] & X\ar[d]\\
\dpp{n-1}Z\ar[r]\sb{\dpp{n-1}r} & \dpp{n-1}\p{n-1}X
}
\noindent in \w[.]{\funcat{n-1}{\Cat}}
Then \w{P\in\Gcatwg{n}} and $w$ is an $n$-equivalence.
Supposing that in addition \w{X\in\Gpdwg{n}} and \w[,]{Z\in\Gpdwg{n-1}}
then \w[.]{P\in\Gpdwg{n}}
\end{proposition}

\begin{proof}
Since \w[,]{X\in\Gcatwg{n}} \w[,]{\p{n-1}X=q\up{n-1}X} so we can apply
\cite[Theorem 9.2.4]{PaolS} and \cite[Corollary 9.2.5]{PaolS} to conclude that
\w{P\in\catwg{n}} and $w$ is an $n$-equivalence. By \cite[Lemma 12.3.4]{PaolS},
it follows that \w[.]{P\in\Gcatwg{n}}
If \w{X\in\Gpdwg{n}} and \w[,]{Z\in\Gpdwg{n-1}} it follows that \w[.]{P\in\Gpd\sp{n}}
Since also \w[,]{P\in\Gcatwg{n}} by definition \w[.]{P\in\Gpdwg{n}}
\end{proof}

%
%
\sect{Comonad resolution of $n$-track categories}
\label{comores}

We now introduce our model for categories enriched in $n$-types.
Although we could use the category $\Gcatwg{n}$ for this purpose, we choose to
work with the smaller category \w{\Gpdwg{n}} and thus define our models (called $n$-track
categories) to be categories enriched in \w[.]{\Gpdwg{n}}
This choice is justified by the fact (see \cite[Theorem 4.32]{BP}) that any $n$-type
is weakly equivalent to the classifying space of a weakly globular $n$-fold groupoid.

The advantage of \w{\Gpdwg{n}} over \w{\Gcatwg{n}} is that it allows us to build
a comonad on $n$-track categories in \S \ref{scntc} below.
\begin{definition}\label{def1-comores}
An \emph{$n$-track category} is a category enriched in the category \w{\Gpdwg{n}} of
weakly globular $n$-fold groupoids (with respect to the cartesian monoidal structure),
and the category \w{\Trk{n}=\hCo{\Gpdwg{n}}} is the category of all (small)
$n$-track categories.
\end{definition}

When \w[,]{n=1} \w{\Trk{1}=\Trt} is the category of track categories,
as in \cite{BP2019}. There is an isomorphism
\begin{myeq}[\label{eq1-comores}]
  \hCo{(\Gpdn{n},\times)}\cong \Gpdn{n}(\Cato)
\end{myeq}
The inclusion \w{\Gpdwg{n}\hookrightarrow \cat{n}} induces an inclusion
$$
  \Trk{n}=\hCo{(\Gpdwg{n},\times)}\hookrightarrow \hpCo{\Gpdn{n},\times}~.
$$

Composing this with the isomorphism \wref{eq1-comores} yields a fully faithful functor
$$
  \Trk{n}\hookrightarrow \Gpdn{n}(\Cato).
$$
We denote the essential image of this functor by \w[.]{\Gpdwg{n}(\Cato)} We can
therefore identify \w{\Trk{n}} with \w[.]{\Gpdwg{n}(\Cato)} This identification
  will be useful in the next section when defining a comonad on \w[.]{\Trk{n}}

\begin{mysubsection}{Comonad on $n$-track categories}
\label{scntc}
By definition of \w[,]{\Gpdwg{n}} there is an inclusion
\w[,]{\Gpdwg{n}\hookrightarrow \Gpd(\Gpdwg{n-1})} and thus there is an arrow
functor \w[,]{u\sb{n-1}:\Gpdwg{n}\to\Gpdwg{n-1}} which in turn induces a functor
\begin{myeq}[\label{eq2-comores}]
U\sb{[n-1]}:\Trk{n}\to\Trk{n-1}\;.
\end{myeq}
\end{mysubsection}

\begin{lemma}\label{lem1-comores}
The functor \w{U\sb{[n-1]}} as above has a left adjoint
$$
L\sb{[n-1]}:\Trk{n-1}\to\Trk{n}~.
$$
\end{lemma}

\begin{proof}
Let \w{u\sb{n-1}} denote the restriction to \w{\Gpdwg{n}} of the internal arrow functor
$$
u\sb{n-1}:\Gpd(\Gpdwg{n-1})\to\Gpdwg{n-1}\;.
$$
By \cite[\S 2.2]{BP2019},  this has a left adjoint
$$
\elll{n-1}:\Gpdwg{n-1}\rw \Gpd(\Gpdwg{n-1})~,
$$
given, for each \w[,]{X\in\Gpdwg{n-1}} by
$$
\elll{n-1}X=H(X\sb{s}\cop X_t \rightleftarrows X)~.
$$

Here $H$ is the internal equivalence relation corresponding to the map $f$
(see \cite{BP2019} for details). We claim that \w{\elll{n-1}X} is in \w[,]{\Gpdwg{n}}
since it satisfies the hypotheses of Proposition \ref{pro02}:
it is a levelwise equivalence relation, so both hypotheses a) and b) hold.
Further, by definition of $H$, \w[,]{\p{n-1}\elll{n-1}X=X\in\Gpdwg{n-1}}
so that hypothesis c) also holds. Also \w[,]{\elll{n-1}X\in\Gpdn{n}}
hence by the Proposition \ref{pro02} \w[.]{\elll{n-1}X\in\Gpdwg{n}}
\end{proof}

By repeated application of Lemma \ref{lem1-comores} we obtain adjunctions
\begin{myeq}[\label{eq3-comores}]
U\sb{n}=U\sb{[n-1]}\cdots U\sb{[1]} U\sb{[0]}:\xymatrix{\Trk{n}\ar@<-1ex>[r] &
    \Cato \ar@<-1ex>[l]}:L\sb{n}=L\sb{[0]}L\sb{[1]}\cdots L\sb{[n-1]}\;.
\end{myeq}
On the other hand, if \w{\Gpho} is the category of graphs with object set $\clO$ and
maps which are identity on $\clO$, there is a free-forgetful functor adjunction
\begin{myeq}[\label{eq4-comores}]
V:\Cato \leftrightarrows \Gpho : F ~.
\end{myeq}
Composing the functors \wref{eq3-comores} and \wref[,]{eq4-comores} we obtain
the adjunction
\begin{myeq}[\label{eq5-comores}]
  L\sb{n} F:\Trk{n} \leftrightarrows \Gpho : V U\sb{n}
\end{myeq}
We therefore have a comonad \w{(\ck,\zve,\zd)} given by
\begin{myeq}[\label{eq6-comores}]
  \ck=L\sb{n} FVU\sb{n}:\Trk{n}\to\Trk{n}~,
\end{myeq}
with $\zve$ is the counit of the adjunction \wref[,]{eq5-comores}
\w[,]{\zd=L\sb{n}F(\eta)VU\sb{n}} and $\eta$ the unit of the adjunction
\wref[.]{eq5-comores}

\begin{mysubsection}{Properties of the comonad resolution}
\label{procomres}
Using the comonad $\ck$ as in \eqref{eq6-comores}, for each \w{X\in\Trk{n}}
we obtain a simplicial object \w{\ckb X\in[\Dop,\Trk{n}]} with
\w[,]{\ck\sb{n} X=\ck\sp{n+1}X} and face and degeneracy maps given by
\w{\pt\sb{i}=\ck^i\zve \ck\sp{n-i}:\ck\sp{n+1}X\rw \ck\sp{n} X} and
\w[.]{\zs^i=\ck^i\zd \ck\sp{n-i}:\ck\sp{n+1}X\rw \ck\sp{n+2}X}
The simplicial object $\ckb X$ is augmented over $X$ via
\w[,]{\zve:\ckb X\rw X} and \w{\ckb X} is a simplicial resolution of $X$
(see \cite[\S 8.6]{WeibH}).

The $n$-fold nerve functor \w{\Gpdwg{n}\to\funcat{n}\Set} induces a functor
\w[,]{\Ne{n}:\Trk{n}\to\funcat{n}{\Cato}} and therefore a functor
$$
\oN{n}:\funcat{}{\Trk{n}}\to\funcat{n+1}{\Cato}
$$
\noindent obtained by applying\w{\Ne{n}} in each simplicial dimension (see notational convention \ref{snac}).

The nerve functor \w{N:\Cato\to\funcat{}{\Set}} also induces a functor
\w[.]{\oN{}:\funcat{n+1}{\Cato}\to\funcat{n+2}{\Set}}
By composition, we therefore obtain the functor
$$
\oN{}\,\oN{n}:\funcat{}{\Trk{n}}\to\funcat{n+2}{\Set}\;.
$$

Given \w[,]{X\in\Trk{n}}
\w{\oN{}\,\oN{n}\ckb X\in\funcat{n+2}{\Set}} has a
simplicial ``resolution'' direction given by the comonad resolution ,
a simplicial ``category'' direction given by the nerve of \w[,]{\Cato}
and $n$ simplicial ``groupoid'' directions given by the nerves of the groupoids in
\w[.]{\Gpdwg{n}}
By a slight abuse, we sometimes identify
\w{\oN{}\,\oN{n}\ckb X\in\funcat{n+2}{\Set}} with
\w[.]{\oN{n}\ckb X\in\funcat{n+1}{\Cato}}

Composing \w{\Ne{n}} with the diagonal functor we obtain the functor
\begin{equation*}
\dI=\diag\, \Ne{n}:\Trk{n}\to\funcat{}{\Cato}=\SOC\;.
\end{equation*}
This functor collapses all groupoid directions in an $n$-track category
into a single simplicial direction. Applying $\dI$ levelwise we therefore obtain
\begin{equation*}
  \odI:\funcat{}{\Trk{n}}\rw[\Dop,\funcat{}{\Cato}]\cong \funcat{2}{\Cato}\;.
\end{equation*}
Composing $\odI$ with the diagonal functor
\begin{equation*}
  \diag:\funcat{2}{\Cato}\to\funcat{}{\Cato}=\SOC\;,
\end{equation*}
yields the functor
\begin{myeq}\label{eq1-procomres}
  \Ovdia=\diag\circ\odI:\funcat{}{\Trk{n}}\rw \SOC\;.
\end{myeq}
Given \w[,]{X\in\Trk{n}} let
\begin{equation*}
W=\oN{}\,\odI\ck\sb{\sbl}X\in\funcat{3}{\Set}\;.
\end{equation*}
Below is a picture of the corner of $W$; in this picture the horizontal simplicial
direction is given by the comonad resolution, the vertical simplicial direction
is given by the diagonal of the multinerve of the weakly globular $n$-fold groupoid
in each track category and the diagonal simplicial direction is given by the nerve
of the category in each track category.
\begin{myeq}\label{eq2-procomres}
\begin{gathered}
  \includegraphics[width=10cm]{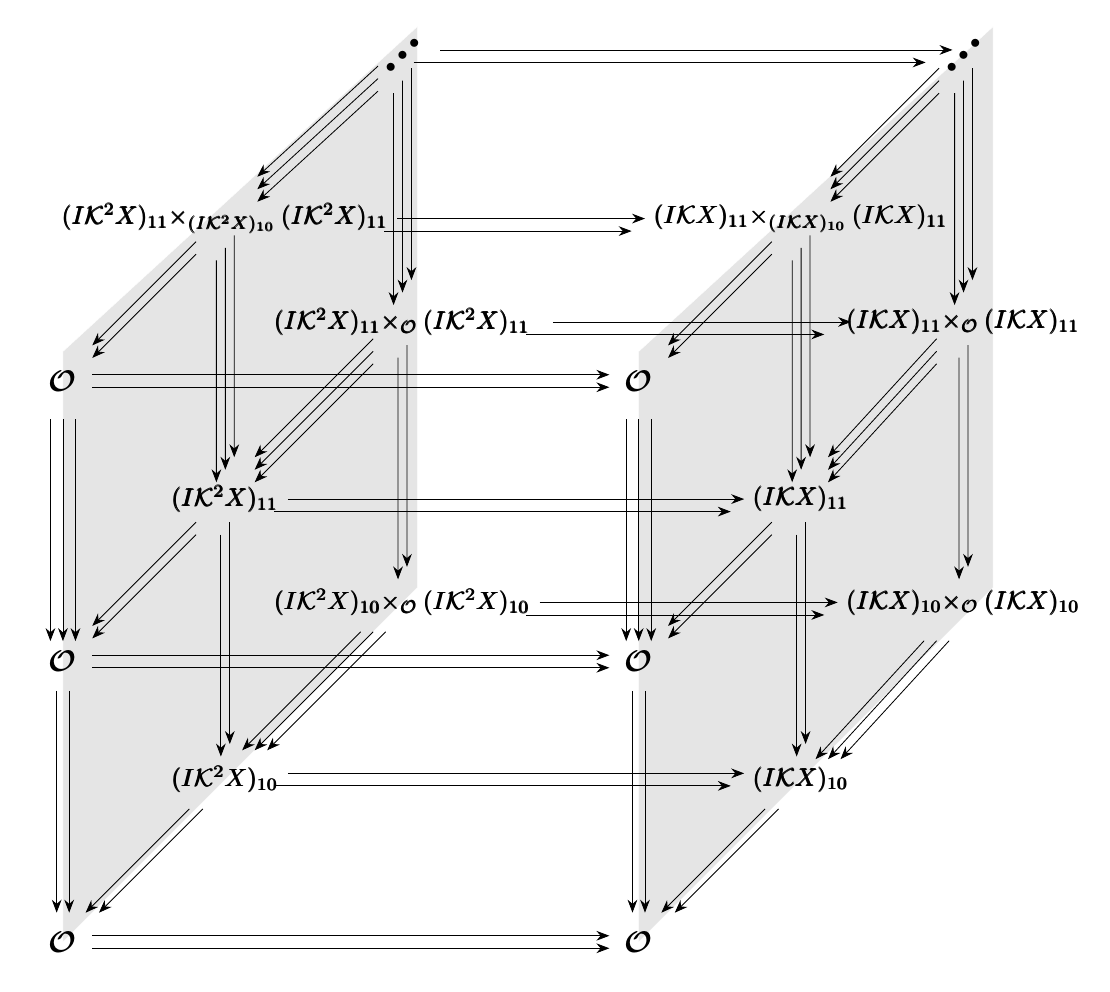}
\end{gathered}
\end{myeq}
The augmentation map \w{\zve:\ckb X\rw X} induces a map
\w{W\rw \oN{}\,\odI c X} in \w[.]{\funcat{3}{\Set}}
Let \w{Z=W\sp{(2)}} be $W$ thought of a simplicial object in \w{\funcat{2}{\Set}}
along the direction appearing diagonal in the picture, that is
\begin{myeq}\label{eq3-procomres}
  Z\in [\Dop,\funcat{2}{\Set}]
\end{myeq}
with \w{Z\sb{0}} the constant bisimplicial set on $\clO$, \w{Z\sb{1}} is given by
\myodiag[\label{eq4-procomres}]{
\vdots & \vdots & \vdots\\
(\dI\ck^3 X)\sb{12} \ar@<1ex>[r] \ar@<0.0ex>[r]  \ar@<-1ex>[r]
\ar@<1ex>[dd] \ar@<0.0ex>[dd]  \ar@<-1ex>[dd]
& (\dI\ck\sp{2} X)\sb{12} \ar@<1ex>[dd] \ar@<0.0ex>[dd]
\ar@<-1ex>[dd] \ar@<0.5ex>[r] \ar@<-0.5ex>[r]
& (\dI\ck X)\sb{12} \ar@<1ex>[dd] \ar@<0.0ex>[dd]  \ar@<-1ex>[dd]\\ \\
(\dI\ck^3 X)\sb{11} \ar@<1ex>[r] \ar@<0.0ex>[r]  \ar@<-1ex>[r] \ar@<0.5ex>[dd]
\ar@<-0.5ex>[dd]
&  (\dI\ck\sp{2} X)\sb{11} \ar@<0.5ex>[r] \ar@<-0.5ex>[r] \ar@<0.5ex>[dd] \ar@<-0.5ex>[dd]
& \ar@<0.5ex>[dd] \ar@<-0.5ex>[dd] (\dI\ck X)\sb{11}\\ \\
 (\dI\ck^3 X)\sb{10} \ar@<1ex>[r] \ar@<0.0ex>[r]  \ar@<-1ex>[r]
&  (\dI\ck\sp{2} X)\sb{10} \ar@<0.5ex>[r] \ar@<-0.5ex>[r]
& (\dI\ck X)\sb{10}
}
and \w{Z\sb{k}\cong \pro{Z\sb{1}}{Z\sb{0}}{k}} for each \w[.]{k\geq 2}
By applying the diagonal functor \w{\diag:\funcat{2}{\Set}\rw \funcat{}{\Set}}
dimensionwise to $Z$ (viewed as in \eqref{eq3-procomres}) we obtain
\begin{equation*}
\ovdia Z=N\Ovdia\ckb  X\in[\Dop,\funcat{}{\Set}]
\end{equation*}
\noindent for \w{\Ovdia} as in \wref[.]{eq1-procomres}
Note that \w{\ovdia Z} is an \wwb{\SO}category: in fact, \w{(\ovdia Z)\sb{0}} is the
constant simplicial set at $\clO$. Since \w{\diag} preserves limits, by
\wref{eq4-procomres} for each \w{k\geq 2} we obtain
\begin{equation*}
\begin{split}
  (\ovdia Z)\sb{k} & \cong\diag\, Z\sb{k}\cong
  \pro{\diag\,Z\sb{1}}{\diag\,Z\sb{0}}{k}= \\
    & =\pro{(\ovdia Z)\sb{1}}{(\ovdia Z)\sb{0}}{k}\;.
\end{split}
\end{equation*}
Therefore, \w{\ovdia Z} is an \wwb{\clS,\clO}category. By a slight abuse of notation,
we identify \w{\ovdia Z} with \w[.]{\Ovdia \ckb X}

The augmentation map \w{\zve:\ckb X\rw X} induces a map \w{\za:\Ovdia\ckb X \rw \dI X}
of \wwb{\clS,\clO}categories.
\end{mysubsection}

\begin{lemma}\label{lem1-procomres}
For  \w[,]{X\in\Trk{n}} the map  \w{\za:\Ovdia\ckb X \rw \dI X}
is a Dwyer-Kan equivalence in \w[.]{\SOC}
\end{lemma}

\begin{proof}
Since \w[,]{\za\sb{0}=\Id} we need to show that for each \w{a,b\in\clO} the map
\begin{myeq}\label{eq5-procomres}
  \za(a,b):(\Ovdia\ckb X)(a,b)\rw(\dI X)(a,b)
\end{myeq}
is a weak homotopy equivalence. This map is the diagonal of the following map
$$
  \zb(a,b):(\ckb X)\sb{1}(a,b)\rw (cX)\sb{1}(a,b)
$$
\noindent in \w[.]{[\Dop,\funcat{n}{\Set}]}
To show that \w{\diag\,\zb(a,b)=\za(a,b)} is a weak equivalence, it is enough
to show that \w{\zb(a,b)} is a levelwise weak equivalence of simplicial sets
in each multi-simplicial dimension: that is, for each
\w{J=(j\sb{1},\ldots,j\sb{n})} in \w[,]{\Dnop} the map
\begin{myeq}\label{eq6-procomres}
  (\ckb X)\sb{1J}(a,b)\rw(cX)\sb{1J}(a,b)
\end{myeq}
is a weak equivalence of simplicial sets.
We distinguish various cases:
\begin{itemize}
\item [a)] If \w[,]{J=(1,\overset{n}{\ldots},1)} then \w{UX=X\sb{J}}
  by definition of $U$. Therefore, \w[,]{X\sb{1J}(a,b)=(VUX)\sb{1}(a,b)}
and thus \w[.]{(\ckb X)\sb{1J}(a,b)=(VU\ckb X)\sb{1}(a,b)}

The augmented simplicial object \w{VU(\ckb X)\xrw{VU\zve}VUX} is aspherical
(see \cite[Prop. 8.6.10]{BP2019}), so the same is true for
\w[.]{(VU\ckb X)\sb{1}(a,b)\rw (VUX)\sb{1}(a,b)}

Thus by the above \w{(\ckb X)\sb{1J}(a,b)\rw (cX)\sb{1J}(a,b)}
is a weak equivalence.
\item [b)] If \w{J=(j\sb{1},\ldots,j\sb{n})\in\Dnop}
with \w{j\sb{i}\in\{0,1\}} for \w[,]{i=1,\ldots,n}
we proceed by descending induction on the number $k$ of digits $1$ in $J$.

When \w[,]{k=0} we are in case a), so we may assume
$$
J\sb{1}\up{i}~=~(j\sb{1},\ldots,j\sb{i-1},1,j\sb{i+1},\ldots,j\sb{n})
$$
\noindent has \w{k\geq 1} digits $1$. Suppose by induction that
\w{(\ckb X)\sb{1J\sb{1}\up{i}}(a,b)\rw (cX)\sb{1J\sb{1}\up{i}}(a,b)}
is a weak equivalence.
Let \w{J\sb{0}\up{i}=(j\sb{1},\ldots,j\sb{i-1},0,j\sb{i+1},\ldots,j\sb{n})}
The number of digits $1$ in \w{J\sb{0}\up{i}} is \w[.]{k-1}
From the commuting diagram (given by the face and degeneracy maps)
\begin{myeq}\label{eq7-procomres}
\begin{gathered}
\xymatrix@C=50pt{
  (\ck\sp{\sbl}X)\sb{1J\sb{1}\up{i}}(a,b)\ar[r] \ar@<1ex>[d] &
  (cX)\sb{1J\sb{1}\up{i}}(a,b) \ar@<1ex>[d]\\
  (\ck\sp{\sbl}X)\sb{1J\sb{0}\up{i}}(a,b)\ar[r] \ar@<1ex>[u] &
  (cX)\sb{1J\sb{0}\up{i}}(a,b) \ar@<1ex>[u]
}
\end{gathered}
\end{myeq}
we see that the bottom map in \wref{eq7-procomres} is a retract of the top map.
Since the latter is a weak equivalence by induction hypothesis we deduce that
\w{(\ckb X)\sb{1J\sb{0}\up{i}}(a,b)\rw (cX)\sb{1J\sb{0}\up{i}}(a,b)}
is a weak equivalence. This proves the inductive step, so we conclude that
\w{(\ckb X)\sb{1J}(a,b)\rw (cX)\sb{1J}(a,b)}
is a weak equivalence for any \w{J=(j\sb{1},\ldots,j\sb{n})\in\Dnop} with each
\w{j\sb{i}} in \w[.]{\{0,1\}}
\item [c)] Let \w{J=(j\sb{1},\ldots,j\sb{n})} with each \w{j\sb{i}} in \w[.]{\{1,2\}}
We proceed by increasing induction on the number $k$ of digits $2$ in $J$.
When \w[,]{k=0} we are in case a).

Now assume that \w{J\sb{2}\up{i}=(j\sb{1},\ldots,j\sb{i-1},2,j\sb{i+1},\ldots,j\sb{n})}
has \w{k>0} digits $2$. Then \w{J\sb{1}\up{i}} and \w{J\sb{0}\up{i}} have
\w{k-1} digits $2$, and by inductive hypothesis \wref{eq6-procomres}
is a weak equivalence when \w{J=J\sb{1}\up{i}} and \w[.]{J=J\sb{0}\up{i}}
Since each simplicial direction in $J$ is the nerve of a groupoid, the map
\wref{eq6-procomres} for \w{J=J\sb{2}\up{i}} is the induced map
of pullbacks
\myqdiag[\label{eq8-procomres}]{
  (\ckb X)\sb{1J\sb{1}\up{i}}(a,b)\ar[r]\sp{\pt\sb{0}\sp{\sbl}} \ar[d] &
  (\ckb X)\sb{1J\sb{0}\up{i}}(a,b)\ar[d] &
  (\ckb X)\sb{1J\sb{1}\up{i}}(a,b) \ar[l]\sb{\pt\sb{1}\sp{\sbl}} \ar[d]\\
  (c X)\sb{1J\sb{1}\up{i}}(a,b)\ar[r]\sp{\pt\sb{0}}  &
  (c X)\sb{1J\sb{0}\up{i}}(a,b) & (c X)\sb{1J\sb{1}\up{i}}(a,b) \ar[l]\sb{\pt\sb{1}}
}
We claim that \w{\pt\sb{1}\sp{\sbl}}  and \w{\pt\sb{0}\sp{\sbl}} are fibrations.
In fact, by definition of $\ck$ there is a pullback
$$
\xymatrix@C=40pt{
  (\ckb X)\sb{1J\sb{1}\up{i}}(a,b)\ar[r]\sp{\pt\sb{1}\sp{\sbl}} \ar[d] &
  (\ckb X)\sb{1J\sb{0}\up{i}} (a,b)\ar[d]\sb{\nab}\\
  (\ckb X)\sb{1J\sb{0}\up{i}}(a,b)\ar[r]\sb{\nab}  &
  \pi\sb{0}(\ckb X)\sb{1J\sb{0}\up{i}} (a,b)
}
$$
where $\nab$ is the fold map. This map is a fibration of simplicial sets
(see \cite[p. 903]{BP2019}). Thus \w[,]{\pt\sp{\sbl}\sb{1}} being the pullback
of a fibration, is itself a fibration. Since, by inductive hypothesis,
the vertical maps in \wref{eq8-procomres} are weak equivalences and
the horizontal maps are fibrations, the induced map of pullbacks is a
fibration. Thus \wref{eq6-procomres} holds for \w[,]{J\sb{2}\up{i}}
proving the inductive step.
\item [d)] In that case \w[,]{J=(j\sb{1},\ldots,j\sb{n})} we work by induction on
  the number $k$ of entries \w{j\sb{i}} with \w[.]{j\sb{i}>2} When \w[,]{k=0}
  we are in case c). Let $J$ be such that \w[.]{j\sb{i}>2} Then
$$
X\sb{1J}=\pro{X\sb{1J\sb{1}\up{i}}}{X\sb{1J\sb{0}\up{i}}}{j\sb{i}}
$$
and similarly for \w[.]{(\ckb X)\sb{1J}} By the  induction hypothesis,
\wref{eq6-procomres} is a weak equivalence for \w{J\sb{1}\up{i}} and \w[.]{J\sb{0}\up{i}}
Reasoning as in c), the maps \w{\sb{1J\sb{1}\up{i}}(a,b)\rw X\sb{1J\sb{0}\up{i}}(a,b)}
are fibrations, and similarly for the maps
\w[.]{(\ckb X)\sb{1J\sb{1}\up{i}}(a,b)\rw (\ckb X)\sb{1J\sb{0}\up{i}}(a,b)}
Therefore, the induced map of iterated pullbacks is a weak equivalence.
That is, \wref{eq6-procomres} is a weak equivalence.
\end{itemize}
\end{proof}

\begin{corollary}\label{ccofrepl}
For any \w[,]{X\in\Trk{n}} \w{\Ovdia\ckb X} is a cofibrant replacement for \w{\dI X}
in the Dwyer-Kan model category \w{\SOC} (see \cite[\S 7.1]{DKanS}).
\end{corollary}

\begin{proof}
This follows as in \cite[Lemma 4.6]{BP2019} from the fact that \w{\Ovdia\ckb X} is
a free \ww{\SO}-category (see \S \ref{scntc} above).
\end{proof}

%
%
\sect{Modules}
\label{cmodule}

The coefficients of a cohomology theory must be a suitable type of module,
in the sense due to Beck in his thesis (see \cite{BecT}):

\begin{mysubsection}{Abelian group objects}\label{genset}
An \emph{abelian group object} in a category $\clD$  with finite products is an
object $G$ equipped with
\begin{itemize}
\item [(i)] a \emph{unit} map \w[,]{\zs:\ast\rw G} where $\ast$ is the terminal
object in \w[,]{\clD} with map \w[;]{\rho:G\to\ast}
\item [(ii)] an associative, commutative, and unital
\emph{multiplication map} \w[,]{\mu:G\times G\rw G} and
\item [(iii)] an \emph{inverse map} \w[.]{i:G\rw G}
\end{itemize}

These are required to satisfy the usual axioms:
\begin{itemize}
\item [a)] Associativity:
\w[.]{\mu\circ(\Id\times\mu)=\mu\circ(\mu\times\Id)\colon
M\oset{\zr}{\times}M\oset{\zr}{\times} M\to M}
\item [b)] Commutativity: \w[,]{\mu\circ\tau=\mu:M\oset{\zr}{\times}M\to M} where
\w{\tau=M\times M\to M\times M} is the switch map.
\item [c)] Inverse: \w[,]{\mu\circ(\Id\times i)\circ\zD=\sigma\circ\rho:M\to M}
where \w{\zD:G\rw G\times G} is the diagonal.
\item [d)] Zero element: \w[.]{\mu\circ(\Id,\zs\zr)=\sigma:M\to M}
\end{itemize}

From these axioms we see that
\begin{myeq}\label{eq2-genset}
\begin{gathered}
\xymatrix@C=40pt{
M\ar[r]^(0.4){\zD} & M\oset{\zr}{\times}M \ar[r]^(0.55){\mu} & M\\
X\ar[u]\sp{\zs} \ar[rru]\sb{\zs}
}
\end{gathered}
\end{myeq}
commutes.
\end{mysubsection}

\begin{mysubsection}{Eilenberg-Mac~Lane objects}\label{eilmacob}
A module $M$ over an object \w{Q\in\GC} is an abelian group object in
\w[.]{(\GC,Q\sb{0})/Q} Thus it is given by a map
\w{\zr:M\rw Q} in \w{\GC} equipped with a section $\zf$ and a multiplication
map \w{\mu: M\oset{\zr}{\times}M \rw M} satisfying the axioms as in \S \ref{genset}.

We now build a groupoid object \w{\EQM{2}\in\GnC{2}} in \w[:]{\GC}
$$
\xymatrix@L=2pt{
  M\oset{\zr}{\times}M \ar[r]\sp{\mu} & M \ar@<2ex>[r]\sp{\zr}\ar[r]\sb{\zr} &
  Q \ar@<2.2ex>[l]\sp{\zf}
}
$$
The fact that this is a groupoid object follows from $M$ being an abelian group object
in \w[.]{(\GC, Q\sb{0})/Q} More explicitly, \w{\EQM{2}} has the form
\myrdiag[\label{eq1-eilmacob}]{
  M\sb{1}\oset{\zr\sb{1}}{\times}M\sb{1} \ar[r]\sp{\mu\sb{1}} \ar@<-1ex>[d]\ar[d] &
  M\sb{1} \ar@<2ex>[r]\sp{\zr\sb{1}}\ar[r]\sb{\zr\sb{1}}\ar@<-1ex>[d]\ar[d]  &
  Q\sb{1}\ar@<-1ex>[d]\ar[d] \ar@<2.2ex>[l]\sp{\zf\sb{1}}\\
Q\sb{0} \ar@<-1ex>[u] \ar@{=}[r] & Q\sb{0} \ar@<-1ex>[u] \ar@{=}[r] &
Q\sb{0} \ar@<-1ex>[u]
}
\end{mysubsection}

\begin{lemma}\label{lem1-eilmacob}
Let \w{\zr:A\rw B} and \w{\zf:B\rw A} be maps in \w{\GC} with
\w{\zr\zf=\Id} and \w[.]{\zr\sb{0}=\Id} Then \w[,]{pB\cong pA} where
\w{p:\GC\rw \clC} is the isomorphism classes of object functor.
\end{lemma}
\begin{proof}
Recall that
\w{pA=\Coeq(\xymatrix{A\sb{1} \ar@<0.5ex>[r]\sp{d\sb{0}}
\ar@<-0.5ex>[r]\sb{d\sb{1}}&A\sb{0}})}
and
\w[.]{pB=\Coeq(\xymatrix{B\sb{1} \ar@<0.5ex>[r]\sp{d'\sb{0}}
    \ar@<-0.5ex>[r]\sb{d'\sb{1}}&B\sb{0}})}
To show that the two coequalizers are isomorphic, consider the diagram
$$
\xymatrix{
  A\sb{1} \ar@<0.5ex>[r]\sp{d\sb{0}}\ar@<-0.5ex>[r]\sb{d\sb{1}} &
A\sb{0} \ar[r]\sp{q} \ar[d]\sb{f} & pB \\
 & C &
}
$$
\noindent in $\clC$, with \w[.]{fd\sb{0}=fd\sb{1}} We have
\w[,]{fd\sb{0}\zf\sb{1}=fd\sb{1}\zf\sb{1}} that is, \w{fd'\sb{0}=fd'\sb{1}}
(since \w{d\sb{i}\zf\sb{1}=d'\sb{1}} for \w[).]{i=0,1}
This implies that there is a map \w{h:pB\rw C} with \w[.]{hq=f} Therefore,
$$
\Coeq(\xymatrix{A\sb{1} \ar@<0.5ex>[r]\sp{d\sb{0}}\ar@<-0.5ex>[r]\sb{d\sb{1}}&
  A\sb{0}})\cong pB\;,
$$
\noindent so  \w[.]{pA\cong pB}
\end{proof}

\begin{corollary}\label{cor1-eilmacob}
For \w{\EQM{2}} as above, \w[,]{\p{1}\EQM{2}=\dd{pQ}}
where \w{\dd{pQ}} is the discrete internal groupoid on \w[.]{pQ}
\end{corollary}
\begin{proof}
For each \w{k\geq 1} there is a split map \w{(\EQM\up{2})\sb{k}\rw Q}
in \w{\GC} which is the identity on objects. Thus by Lemma \ref{lem1-eilmacob},
\w[.]{p(\EQM{2})\sb{k}=pQ=pQ} This shows that \w[.]{\p{1}\EQM{2}=\dd(pQ)}
\end{proof}

\begin{lemma}\label{lem2-eilmacob}
There are maps
$$
\xymatrix{
\EQM{2} \ar@<0.5ex>[r]\sp{\zr\up{2}} & \dpp{1}Q \ar@<0.5ex>[l]\sp{\zf\up{2}}
}
$$
in \w{\Gpd(\GC)} given by
\myrdiag[\label{eq2-eilmacob}]{
  M\oset{\zr}{\times}M \ar[r]\sp{\mu} \ar@<0.5ex>[d]\sp{\zr\mu} &
  M \ar@<2ex>[r]\sp{\zr}\ar[r]\sp{\zr}\ar@<-0.5ex>[d]\sb{\zr}  & Q\ar@<2ex>[l]\sp{\zf}\\
  Q \ar@<0.5ex>[u]\sp{\zD\zf} \ar@{=}[r] & Q \ar@<-0.5ex>[u]\sb{\zf} \ar@{=}[r] &
  Q \ar@{=}[u]
}
satisfying \w[.]{\zr\up{2}\zf\up{2}=\Id}
\end{lemma}

\begin{proof}
We need to check the commutativity of \wref[.]{eq2-eilmacob} By \wref[,]{eq2-genset}
\w[.]{\mu\zD\zf=\zf} All other cases are trivial. Hence \w{\zr\up{2}}
and \w{\zf\up{2}} are maps in \w[.]{\Gpd(\GC)} We also have
\w[.]{\zr\mu\zD\zf=\zr\zf=\Id} Therefore, \w[.]{\zr\up{2}\zf\up{2}=\Id}
\end{proof}

\begin{mysubsection}{Abelian group objects}\label{sago}
From the construction of \w{\EQM{2}} and the axioms of \S \ref{genset},
we see that \w{\EQM{2}} is an abelian group object in \w[.]{(\GnC{2},Q)/\dpp{1}Q}

We next define by induction an abelian group object \w{\EQM{n}} in
$$
(\GnC{n},\dpp{n-1}Q)/\dpp{n}Q\;.
$$
For \w[,]{n=2} \w{\EQM{2}} is as above. Now assume by induction that we have defined
\w{\EQM{k}} for \w[,]{k\leq n-1} together with maps
\begin{equation*}
\xymatrix{
\EQM{k} \ar@<0.5ex>[r]\sp{\zr\up{k}} & \dpp{n}Q \ar@<0.5ex>[l]\sp{\zf\up{k}}
}
\end{equation*}
satisfying \w[.]{\zr\up{k}\zf\up{k}=\Id} Let \w{\EQM{n}} be the object
\myrdiag[\label{eq3-eilmacob}]{
  \xrw{\mu\sb{1}\up{n-1}}\oset{\DV}{(E\sb{\clC}\up{n-1}(Q,M))\sb{1}}
  \ar@<2ex>[r]^(0.65){\zr\sb{1}\up{n-1}}\ar[r]\sb(0.65){\zr\sb{1}\up{n-1}} &
  \dpp{n-1}\oset{\DV}{Q\sb{1}} \ar@<3ex>[l]^(0.35){\zf\sb{1}\up{n-1}}
  \ar@<-1ex>[d]\sb{pr\sb{1}}
  \ar@<0ex>[d]\sp{pr\sb{2}} \\
\cdots \dpp{n-1}Q\sb{0} \ar@{=}[r] & \dpp{n-1} Q\sb{0} \ar@<-3ex>[u]\sb{\zD}
}
\noindent of \w[.]{\Gpd\sp{2}(\Gpd\sp{n-2}\clC)=\GnC{n}}
There are maps
$$
\xymatrix{
  \EQM{n} \ar@<0.5ex>[r]\sp{\zr\up{n}} & \dpp{n}Q \ar@<0.5ex>[l]\sp{\zf\up{n}}
}
$$
with \w{\zr\up{n}\zf\up{n}=\Id} given by
\w[,]{\zr\sb{*0}\up{n}=\zf\sb{*0}\up{n}=\Id}
while \w{\zr\sb{*1}\up{n}} and \w{\zf\sb{*1}\up{n}} are given by
$$
\resizebox{\linewidth}{!}{
\xymatrix@C=40pt@R=40pt@L=2pt{
  (\EQM{n-1})\sb{1}
  \tiund{\dpp{n-1}Q\sb{1}}(\EQM{n-1})\sb{1}
  \ar[r]^(0.7){\mu\sb{1}\up{n-1}} \ar@<0.5ex>[d]\sp{\zD\zf\sb{1}\up{n-1}} &
  (\EQM{n-1})\sb{1} \ar@<3ex>[r]\sp{\zr\sb{1}\up{n-1}}\ar[r]\sp{\zr\sb{1}\up{n-1}}
  \ar@<0.5ex>[d]\sp{\zr\sb{1}\up{n-1}}  &
  \dpp{n-1}Q\sb{1}\ar@<3ex>[l]\sb{\zf\sb{1}\up{n-1}}\\
  \dpp{n-1}Q\sb{1} \ar@<0.5ex>[u]\sp{\zr\sb{1}\up{n-1}\mu\sb{1}\up{n-1}}
  \ar@{=}[r] & \dpp{n-1}Q\sb{1} \ar@<0.5ex>[u]\sp{\zf\sb{1}\up{n-1}} \ar@{=}[r] &
  \dpp{n-1}Q\sb{1} \ar@{=}[u]
}}
$$
The commutativity of this diagram is proved as in Lemma \ref{lem2-eilmacob}.
By the induction hypothesis \w{\EQM{n-1}} is an
abelian group object in \w[;]{(\Gpd\up{n-1}\clC,\dpp{n-1}Q)/\dpp{n}Q}
so by \wref[,]{eq2-genset} \w[,]{\mu\up{n-1}\zD\zf\up{n-1}=\zf\up{n-1}}
which implies that \w[.]{\mu\sb{1}\up{n-1}\zD\zf\sb{1}\up{n-1}=\zf\sb{1}\up{n-1}}
The rest of the proof is as in Lemma \ref{lem2-eilmacob}.

To complete the inductive construction of \w[,]{\EQM{n}} we need to show that it
is an abelian group object in \w[,]{(\Gpd\up{n}\clC,\dpp{n-1}Q)/\dpp{n}Q}
with structure maps given by \w[,]{\zf\up{n}:\dpp{n}Q\rw \EQM{n}}
$$
\xymatrix{
\EQM{n}\tiund{c(Q)}\EQM{n} \ar[rr]\sp{\mu\up{n}}\ar[rd] &&
\EQM{n} \ar[ld]\sp{\zr\up{n}}\\
& \dpp{n}Q~, &
}
$$
\noindent and
$$
\xymatrix{
\EQM{n} \ar[rr]\sp{i\up{n}}\ar[rd]\sb{\zr\up{n}} && \EQM{n} \ar[ld]\sp{\zr\up{n}}\\
& \dpp{n}Q &
}
$$
with \w[,]{(i\up{n})\sb{11}=i\sb{1}\up{n-1}} \w[,]{(i\up{n})\sb{\ast0}=\Id}
and \w[.]{(i\up{n})\sb{01}=\Id} The axioms of abelian group object hold since
they hold for \w{\EQM{n-1}} by the induction hypothesis.
\end{mysubsection}

In the following Lemma we give an explicit description of
\w[,]{\clN\sb{n}\EQM{n}} where \w{\clN\sb{n}:\GnC{n}\to\funcat{n}{\clC}}
is the multinerve.

\begin{lemma}\label{lem3-eilmacob}
For \w{\EQM{n}} as above, let \w[.]{E:=\clN\sb{n}\EQM{n}}
Then \w[,]{E\sb{\oset{n}{1\ldots 1}}=M\sb{1}} \w[,]{E\sb{\sbl 0}=dQ\sb{0}}
and \w{E\sb{\uk}=Q\sb{1}} for \w{\uk\neq(\oset{n}{1\ldots 1})} and
\w[.]{\uk\neq (\sbl,0)}
\end{lemma}

\begin{proof}
  By induction on $n$. When \w[,]{n=2} this holds by definition of \w[.]{\EQM{2}}
Suppose by induction that this holds for \w[.]{k\leq n-1} By definition of
\w{\EQM{n}} we have
$$
  E\sb{\oset{n}{1\ldots 1}}=
  (\clN\sb{n-2}(\EQM{n-1})\sb{1})\sb{\oset{n-2}{1\ldots 1}}=
  \clN\sb{n-1}(\EQM{n-1})\sb{\oset{n-1}{1\ldots 1}}\;.
$$
Since by induction hypothesis
\w[,]{(\clN\sb{n-1}\EQM{n-1})\sb{\oset{n-1}{1\ldots 1}}=M\sb{1}}
it follows that \w[.]{E\sb{\oset{n}{1\ldots 1}}=M\sb{1}}
The rest is straightforward by the definition of \w[.]{\EQM{n}}
\end{proof}

Below is a picture of the corner of \w[.]{\clN\sb{3}\EQM{3}}
\begin{equation*}
\resizebox{5cm}{!}{
\begin{tikzcd}[ampersand replacement=\&]
	\&\& {M\sb{1}} \&\&\& {Q\sb{1}} \\
	{Q\sb{0}} \&\&\& {Q\sb{0}} \\
	\\
	\&\& {Q\sb{1}} \&\&\& {Q\sb{1}} \\
	{Q\sb{0}} \&\&\& {Q\sb{0}}
	\arrow[shift left=1, from=1-3, to=1-6]
	\arrow[from=1-3, to=1-6]
	\arrow[shift left=1, from=1-6, to=1-3]
	\arrow[shift right=1, from=1-3, to=4-3]
	\arrow[from=1-3, to=4-3]
	\arrow[shift right=1, from=4-3, to=1-3]
	\arrow[shift right=1, from=1-6, to=4-6]
	\arrow[from=1-6, to=4-6]
	\arrow[shift right=1, from=4-6, to=1-6]
	\arrow[shift left=1, no head, from=4-3, to=4-6]
	\arrow[shift right=1, no head, from=4-3, to=4-6]
	\arrow[shift right=1, from=1-3, to=2-1]
	\arrow[from=1-3, to=2-1]
	\arrow[shift right=1, from=2-1, to=1-3]
	\arrow[shift right=1, from=4-3, to=5-1]
	\arrow[from=4-3, to=5-1]
	\arrow[shift right=1, from=5-1, to=4-3]
	\arrow[shift right=1, no head, from=2-1, to=5-1]
	\arrow[shift left=1, no head, from=2-1, to=5-1]
	\arrow[shift right=1, from=1-6, to=2-4]
	\arrow[from=1-6, to=2-4]
	\arrow[shift right=1, from=2-4, to=1-6]
	\arrow[shift left=1, no head, from=2-4, to=5-4]
	\arrow[shift right=1, no head, from=2-4, to=5-4]
	\arrow[shift right=1, no head, from=2-1, to=2-4]
	\arrow[shift left=1, no head, from=2-1, to=2-4]
	\arrow[shift right=1, no head, from=5-1, to=5-4]
	\arrow[shift left=1, no head, from=5-1, to=5-4]
	\arrow[shift right=1, from=4-6, to=5-4]
	\arrow[from=4-6, to=5-4]
	\arrow[shift right=1, from=5-4, to=4-6]
\end{tikzcd}}
\end{equation*}
\begin{corollary}\label{cor2-eilmacob}
For \w{n\geq 3} and \w{E\sb{\clC}\up{n}(Q,M)} as above,
\w[.]{\p{n-1}E\sb{\clC}\up{n}(Q,M)=\dpp{n-1}Q}
\end{corollary}
\begin{proof}
From the description of the multinerve of \w{\EQM{n}} given in
Lemma \ref{lem3-eilmacob} it is enough to show that the coequalizer of
$$
\xymatrix{
M\sb{1} \ar@<0.5ex>[r] \ar@<-0.5ex>[r] & Q\sb{1}
}
$$
is given by \w[,]{Q\sb{1}} and for each \w[,]{k\geq 2} the coequalizer of
\begin{equation*}
\xymatrix{
\pro{M\sb{1}}{Q\sb{0}}{k}\ar@<0.5ex>[r] \ar@<-0.5ex>[r] &
\pro{Q\sb{1}}{Q\sb{0}}{k}
}
\end{equation*}
is given by \w[.]{\pro{Q\sb{1}}{Q\sb{0}}{k}} Since there are split maps
\begin{equation*}
\xymatrix{
M\sb{1} \ar@<0.5ex>[r]\ar@<-0.8ex>[d]\ar@<0ex>[d] & Q\sb{1} \ar@<0.5ex>[l]\ar@{=}[d]\\
Q\sb{1} \ar@<-0.8ex>[u] \ar@{=}[r] & Q\sb{1}
}
\end{equation*}
in \w[,]{\GC}, we have split maps
\begin{equation*}
\xymatrix@C=40pt@R=40pt{
  \pro{M\sb{1}}{Q\sb{0}}{k} \ar@<0.5ex>[r]\ar@<-0.8ex>[d]\ar@<0ex>[d] &
  \pro{Q\sb{1}}{Q\sb{0}}{k} \ar@<0.5ex>[l]\ar@{=}[d]\\
\pro{Q\sb{1}}{Q\sb{0}}{k} \ar@<-0.8ex>[u] \ar@{=}[r] & \pro{Q\sb{1}}{Q\sb{0}}{k}
}
\end{equation*}
by Lemma \ref{lem1-eilmacob}.
\end{proof}

\begin{lemma}\label{lem4-eilmacob}
Let \w{f:W\rw Q} be a map in \w{\GC} and let $M$ be a $Q$-module.
Let \w{f\sp{\ast}M} be the $W$-module given by the pullback
$$
\xymatrix{
f\sp{\ast} M \ar[r]\ar[d] & M \ar[d]\\
W\ar[r]\sb{f} & Q
}
$$
in \w[.]{\funcat{}{\clC}} Then there is a pullback
\begin{myeq}\label{eq4-eilmacob}
\xymatrix{
E\up{n}\sb{\clC}(W,f\sp{\ast}M) \ar[r]\ar[d] & E\up{n}\sb{\clC}(Q,M) \ar[d]\\
\dpp{n}W \ar[r]\sb{\dpp{n}f} & \dpp{n}Q
}
\end{myeq}
in \w[.]{\funcat{n}{\clC}}
\end{lemma}

\begin{proof}
The pullback \wref{eq4-eilmacob} is obtained by replacing \w{Q\sb{i}} by
\w{W\sb{i}} \wb[,]{i=0,1} and \w{M\rw Q} with \w{f\sp{\ast}M\rw W} in the
definition of \w[.]{E\up{n}\sb{\clC}(Q,M)} By construction, this yields
\w[.]{E\up{n}\sb{\clC}(W,f\sp{\ast}M)}
\end{proof}

We finally specialize our constructions to the case \w[.]{\clC=\Cato}
In what follows we abbreviate \w{E\up{n}\sb{\Sz{\Cato}}(Q,M)} by \w[.]{\EQM{n}}

\begin{proposition}\label{pro1-eilmacob}
For each \w{n\geq 2} and \w{E\up{n}(Q,M)} be as above:
\begin{itemize}
\item [a)] \w{E\up{n}(Q,M)} is in \w[.]{\hy{\Gpd\sp{n}}{\Cato}}
\item [b)] \w{E\up{n}(Q,M)} is an abelian group object in
\w[.]{(\hy{\Gpd\sp{n}}{\Cato,\dpp{n-1}Q})/\dpp{n}Q}
\item [c)] \w{E\up{n}(Q,M)} is an Eilenberg-Mac~Lane object in
\w[.]{\hy{\Gpd\sp{n}}{\Cato}}
\end{itemize}
\end{proposition}

\begin{proof}
(a) and (b) follow from the construction in \S \ref{sago}, while (c) follows from
Lemma \ref{lem3-eilmacob} (see \cite[1.3(iv)]{DKSmitO}).
\end{proof}

%
%
\sect{Andr\'{e}-Quillen cohomology of $n$-track categories}
\label{ccohtrk}

We are finally in a position to state and prove our main result, which identifies
the Andr\'{e}-Quillen cohomology groups of an $n$-track category \wh and thus those
of an \wwb{\infty,1}category in dimensions $\leq n$ \wh with a certain \w{algebraic}
cohomology defined in terms of an explicit cosimplicial abelian group
(or cochain complex) constructed through our comonad resolution.

This is deduced from a certain long exact sequence, extending that obtained in
\cite[\S 5]{BP2019} (see \wref{eqlongseq} below).

\begin{mysubsection}{Andr\'{e}-Quillen cohomology}
\label{sdkcohntc}
For \w{X\in\Trk{n}} and $M$ a module over \w[,]{\p{1}X\in\Trt}
let \w{\EX{M}{n}} be the corresponding twisted Eilenberg-Mac~Lane
$\SO$-category (see \cite[\S 1.3(iv)]{DKSmitO}). The \emph{Andr\'{e}-Quillen cohomology}
of $X$ with coefficients in $M$ is given by
$$
\HAQ{n-i}(\dI X,M)~:=~
\pi\sb{i}\,\map\sb{\SOC/\dI X}(\Diag\ovl{\clF}\sb{\sbl}\dI X,\,\EX{M}{n})~,
$$
where \w{\Diag\ovl{\clF}\sb{\sbl}\dI X\to\dI X} is the Dwyer-Kan standard free
resolution.  As noted in the introduction, this was originally introduced in
\cite{DKanO,DKSmitO}, where it was termed $\SO$-\emph{cohomology}, but in light of
\cite[Proposition 2.3]{HNPrasA}, we shall stick with the more common used name of
\emph{Andr\'{e}-Quillen cohomology}.

Since by Corollary \ref{ccofrepl}, \w{\ovl{\Diag}\ckb X\rw IX} is
a cofibrant replacement, as in the proof of \cite[Proposition 4.8]{BP2019}
we deduce that
\begin{myeq}\label{eq1-longseq}
\HAQ{n-i}(\dI X,M)~=~\pi\sb{i}\,\map\sb{\SOC/\dI X}(\Diag\ckb X,\EX(M,n))~.
\end{myeq}
\end{mysubsection}

\begin{lemma}\label{rem1-longseq}
For any \w[,]{X\in\Cato} \w{\clL\sb{n} X} is a homotopically discrete track
category with \w[.]{\p{0}\clL\sb{n} X=X} Moreover, for each \w{Y\in\Trk{n}} and
\w[,]{\ur\in\Dnop} \w{(\ck Y)\sb{\ur}} is a free category.
\end{lemma}

\begin{proof}
By induction on $n$: for the case \w[,]{n=1} see \cite[Remark 4.4]{BP2019}.
Suppose the claim holds for \w[.]{n-1} By definition of $\ck$,
\w[.]{\ck Y=L\sb{[n-1]}Z=H(Z\cop Z\rightleftarrows Z)}
Thus \w[.]{(\ck Y)\sb{\ur}=[H(Z\cop Z\rightleftarrows Z)]\sb{\ur}}
By the induction hypothesis, \w{Z\sb{\ur}} is a free category and by the form of
$H$ (see \cite[\S 2.2]{BP2019}), \w{(\ck Y)\sb{\ur}} is given by coproducts of
copies of \w[,]{Z\sb{\ur}} and is therefore a free category.
\end{proof}

\begin{proposition}\label{pro1-longseq}
Given \w{X\in\Trk{n}} a module $M$ over \w[,]{\p{1}X\in\Trt} and \w[,]{s\geq 0}
The $s$-th cohomology group \w{\HAQ{s}(\dI X;M)} is isomorphic to
\w[,]{\pi\sp{s} C\sp{\sbl}} where \w{C\sp{\sbl}} is the cosimplicial abelian group
\w[.]{\pi\sb{1}\map\sb{\SOC/X}(\odI\ckb X,M)}
\end{proposition}

\begin{proof}
The proof is formally analogous to that of \cite[Theorem 4.9]{BP2019},
using the fact that, by Remark \ref{rem1-longseq}, \w{\dI\ck\sp{s} X} is a cofibrant
\wwb{\SO}category, since it is free in each dimension and the degeneracy
maps take generators to generators. Also by Lemma \ref{rem1-longseq},
\w{\dI\ck\sp{s} X} is homotopically discrete, so
\w{\dI\ck\sp{s} X\to\Id\p{0}\ck\sp{s} X} is a weak equivalence.
The rest of the proof is identical with that of \cite[Theorem 4.9]{BP2019}.
\end{proof}

\begin{lemma}\label{lem1-longseq}
For any map \w{f:A\rw B} in \w{\Cato} with $A$ free and $M$ a $B$-module,
  \w[,]{Z=\clL\sb{n} A} there is an isomorphism
$$
  \pi\sb{n}\map\sb{\SOC/B} (\dI Z,\EB{M}{n})
  =\HAQB{0}(A;M)\cong\Hom\sb{\Sz{\Cato}/B}(A,M\sb{1})\;.
$$
\end{lemma}
\begin{proof}
By Lemma \ref{rem1-longseq}, \w{\dI Z\rw A} is a cofibrant replacement, hence
$$
\pi\sb{n}\map\sb{\SOC/B} (\dI Z,\EB{M}{n})~=~\HAQB{0}(A,M)\;.
$$
The second isomorphism is \cite[Lemma 5.1]{BP}.
\end{proof}

\begin{lemma}\label{lem2-longseq}
Given \w{X\in\Trk{n}} for \w[,]{n\geq 2} there exists \w{S(X)\in\Trk{n}} with
\w{\p{0}(S(X))\sb{0}\in\Cato} a free category, and an $n$-equivalence
\w[.]{v\sb{X}:S(X)\rw X}
\end{lemma}

\begin{proof}
By induction on $n$:

For \w{X\in\Trk{2}} we have \w[,]{\p{1}X\in\Trk{}} so by \cite{BP} there exists a
$2$-equivalence \w{v:Z(X)\to\p{1}X} in \w[,]{\Trk{}} where \w{Z(X)\in\Trk{}} has
\w{(Z(X))\sb{0}} a free category.
Consider the following pullback in \w[:]{[\Dop,\Gpd(\Cato)]}
$$
\xymatrix{
S(X)\ar[r]\sp{v\sb{X}}\ar[d] & X \ar[d]\\
d\sp{(1)}Z(X)\ar[r]\sb{d\sp{(1)}v} & d\sp{(1)}\p{1}X~.
}
$$
We claim that \w{S(X)} is in\w[.]{\Trk{2}}
In fact, for each \w{k\geq 0} there is a pullback
$$
\xymatrix{
(S(X))\sb{k}\ar[r]\ar[d] & X\sb{k} \ar[d]\\
d (Z(X))\sb{k}\ar[r] & d p X\sb{k}
}
$$
\noindent in \w[.]{\Gpd(\Cato)}

Since $p$ commutes with pullbacks over discrete objects, we deduce that
\w[,]{p(S(X))\sb{k}=(Z(X))\sb{k}} and therefore
\begin{myeq}\label{eq2-longseq}
\p{1}S(X)=Z(X)\;.
\end{myeq}
Since \w{X\sb{0}} is homotopically discrete, while \w{d(Z(X))\sb{0}} and
\w{dpX\sb{0}} are discrete, it follows that \w{(S(X))\sb{0}} is
homotopically discrete. By \wref{eq1-longseq} \w{\p{1}S(X)} is in \w[,]{\Gpd(\Cato)}
so by Proposition \ref{pro02} \w[.]{S(X)\in\Trk{2}}

By construction, for each \w[,]{a,b\in (S(X))\sp{d}\sb{0}=(Z(X))\sb{0}}
there is a pullback
$$
\xymatrix{
(S(X))\sb{1}(a,b)\ar[r]\sp{v\sb{X}(a,b)} \ar[d] & X\sb{1}(va,vb)\ar[d]\\
\dd(Z(X))\sb{1}(a,b)\ar[r]\sb{dv(a,b)} & \dd(pX\sb{1})(va,vb)
}
$$
\noindent in \w[.]{\Gpd(\Cato)}

Since $v$ is an equivalence in \w[,]{\Gpd(\Cato)} \w{v(a,b)} is a bijection.
Therefore, \w{v\sb{X}(a,b)} is an isomorphism. Together with the fact that
\w{\p{1}v\sb{X}=v} is an equivalence in \w[,]{\Trk{}} this shows that
\w{v\sb{X}} is a $2$-equivalence in \w[.]{\Trk{2}}
By construction, \w{\p{0}(S(X))\sb{0}=(Z(X))\sb{0}} is a free category.

Suppose by induction that the Lemma holds for \w[,]{n-1} and let
\w[.]{X\in\Trk{n}} Then \w{\p{n-1}X} is in \w[,]{\Trk{n-1}} so by the
inductive hypothesis there are \w{S(\p{n-1}X)\in\Trk{n-1}} (such that
\w{\p{0}[S(\p{n-1}X]\sb{0}} is a free category) and an \wwb{n-1)}equivalence
$$
S(\p{n-1}X)\xrw{\nu\sb{\p{n-1}X}} \p{n-1}X\;.
$$
Consider the following pullback
\begin{equation*}
\xymatrix@C=60pt{
S(X)\ar[r]\sp{\nu\sb{X}}\ar[d] & X\ar[d]\\
\dpp{n}S(\p{n-1}X)\ar[r]\sb{\dpp{n}\nu\sb{\p{n-1}X}} & \dpp{n}\p{n-1}X
}
\end{equation*}
\noindent in \w[.]{\funcat{n-1}{\Gpd(\Cato)}}

We claim that \w{S(X)} is in \w[.]{\Trk{n}} In fact, for each
\w[,]{\uk\in\zD\sp{{n-1}\op}}there is a pullback
\begin{equation*}
\xymatrix@C=40pt{
[S(X)]\sb{\uk}\ar[r]\ar[d] & X\sb{\uk}\ar[d]\\
d S(\p{n-1}X)\sb{\uk}\ar[r] & d(\p{n-1}X)\sb{\uk}
}
\end{equation*}
\noindent in \w[.]{\Gpd(\Cato)}

Since \w[,]{X\in\Trk{n}} \w{X\sb{0}\sp{\{r\}}} is a levelwise equivalence relation;
as \w{\dd\p{1}X} and \w{dZ(X)} are discrete, it follows that \w{S(X)\sb{0}\sp{\{r\}}}
is a levelwise equivalence relation. Since, by \wref[,]{eq2-longseq}
\w{\p{n-1}S(X)} is in \w[,]{\Trk{n-1}} it follows by Proposition \ref{pro02} that
\w[.]{S(X)\in\Trk{n}} By construction, for each
\w{a,b\in S(X)\sb{0}\sp{d}=[S(\p{n-1}X]\sb{0}\sp{d}} there is a pullback
\begin{equation*}
\xymatrix@C=75pt{
(S(X))\sb{1}(a,b) \ar[r]\sp{v\sb{X}(a,b)}\ar[d]& X\sb{1}(va,vb)\ar[d]\\
  d\sp{(n-1)}S(\p{n-1})\sb{1}(a,b)  \ar[r]\sb{d\sp{(n-1)}v\sb{\p{n-1}X}(a,b)} &
  d\sp{(n-1)}\p{n-1}X\sb{1}(va,vb)
}
\end{equation*}
\noindent in \w[.]{\funcat{n-2}{\Cato}}

Since by inductive hypothesis \w{v\sb{\p{n-1}X}(a,b)} is an \wwb{n-1)}equivalence by Proposition
\ref{pro03}, \w{v\sb{X}(a,b)} is an $n$-equivalence. In conclusion, $v\sb{X}$
is an $n$-equivalence. By \wref{eq2-longseq}
\begin{equation*}
  \p{0}(S(X))\sb{0}=\p{n}\p{n-1}(S(X))\sb{0}=\p{0}(S(\p{n-1}X))\sb{0}
\end{equation*}
and by inductive hypothesis this is a free category. This completes the inductive step.
\end{proof}

\begin{proposition}\label{pro2-longseq}
Let $X\in\Trk{n}$ and $M$ be a module over $\p{1}X\in\Trk{}$. Let
\w{d\sb{I}:\clK\sb{s} X\rw X} be the iterated face map. Let
\begin{equation*}
  D\sp{s}=\Hom\sb{\Sz{\Trk{n}}/\clK\sb{s} X}(\clK\sb{s} X,E\up{n}(\p{1}\clK\sb{s}
  X,d\sb{I}\sp{\ast} M))\;.
\end{equation*}
Then $\{D\sp{s}\}\sb{s \geq 0}$ is a cosimplicial abelian group.
\end{proposition}

\begin{proof}
By Lemma \ref{lem4-eilmacob} there is a pullback in $\funcat{n}{\Cato}$
$$
\xymatrix@C=50pt{
  E\up{n}(\p{1}\clK\sb{s} X,d\sb{I}\sp{\ast} M) \ar[r] \ar[d] &
  E\up{n}(\p{1}X, M)\ar[d]\\
\dpp{n}\p{1}\clK\sb{s} X \ar[r]\sb{\dpp{n}\p{1}d\sb{I}} & \dpp{n}\p{1} X
}
$$
Therefore
\begin{equation*}
\begin{split}
  D\sp{s} & = \Hom\sb{\Sz{\Trk{1}}/\p{1}\clK\sb{s} X}
  (\clK\sb{s} X,E\up{n}(\p{1}\clK\sb{s} X,d\sb{I}\sp{\ast} M)) \cong\\
  & \cong \Hom\sb{\Sz{\Gpd\sp{n}(\Cato)}/
    \dpp{n}\p{1}X}(\clK\sb{s} X,E\up{n}(\p{1}X, M))\;.
\end{split}
\end{equation*}
Since \w{E\up{n}(\p{1}X, M)} does not depend on $s$, \w{\{D\sp{s}\}\sb{s \geq 0}}
is a cosimplicial abelian group.
\end{proof}

We now introduce our algebraic model for the cohomology of an
$n$-track category, and show how it fits into a certain long exact sequence relating
it to the Andr\'{e}-Quillen cohomology:

\begin{definition}\label{def1-longseq}
For \w{X\in\Trk{n}} and $M$ a \ww{\p{1}X}-module, let \w{D\sp{\sbl}} be the
cosimplicial abelian group of Proposition \ref{pro2-longseq}.
We define the $n$-th \emph{algebraic cohomology group} of $X$ with coefficients
in $M$ to be \w[.]{\HAlg{n}(X;M):=\pi\sp{s} D\sp{\sbl}}
\end{definition}

In order to construct the long exact sequence, we need to recall some constructions
and results from \cite{BP2019}, which were originally needed only for the case \w{n=2}
of the present paper, but were presented there in a generality sufficient for
our purposes:

\begin{definition}\label{dsplitt}
Let $\clC$ be a category with enough limits, and \w{Z=H(Y)} for
\w{Y=\xymatrix{(Z\sb{0} \ar@<0.5ex>^{q}[r] & \pi\sb{0} \ar@<0.5ex>^{t}[l])}}
in \w[,]{\Spl(\clC)} with \w{\dZ\sb{0}} the discrete internal groupoid on
\w[.]{Z\sb{0}} Let \w{\rho:M\to Z} be a module in
\w[,]{[\Gpd(\clC,Z\sb{0})/Z]\sb{\ab}}  and define \w{j\colon\dZ\sb{0} \rw Z}
to be the map
\mypdiag[\label{eqmapj}]{
Z\sb{0} \ar^(0.4){\zD\sb{Z\sb{0}}}[r] \ar@<-1.0ex>[d]\sb{\Id} \ar[d]\sp{\Id} &
Z\sb{0} \tiund{q}Z\sb{0} \ar@<-1ex>[d]\sb{\pr\sb{0}} \ar[d]\sp{\pr\sb{1}} \\
Z\sb{0} \ar@<-1ex>@/_0.8pc/_{\Id}[u] \ar[r]\sb{\Id\sb{Z\sb{0}}} &
Z\sb{0} \ar@<-1ex>@/_1.1pc/[u]\sb{\zD\sb{Z\sb{0}}}
}
\noindent in \w[,]{(\GC ,Z\sb{0})} and consider the pullback
\mypdiag[\label{sbs-map-zvt-eq1}]{
  j\sp{\ast} M\ar[r]\sp{k} \ar[d]\sb{\zl} & M\ar[d]\sp{\zr} \\
\dZ\sb{0} \ar[r]\sb{j} & Z
}
\noindent in \w[,]{\GC/Z} where
\w[,]{d\sb{0}=d\sb{1}=\zl\sb{1}:(j\sp{\ast} M)\sb{1} \rw Z\sb{0}} since
\w{\dZ\sb{0}} is discrete.
\end{definition}

\begin{proposition}\label{pro3-longseq}
For $\clC$, \w[,]{Y\in\Spl(\clC)} \w[,]{Z=H(Y)} and $M$ a module over $Z$
as in Definition \ref{dsplitt}, there is a short exact sequence of abelian groups
\begin{myeq}\label{eqshortes}
\begin{split}
0~\rw~\Hom\sb{\clC/\pi\sb{0}}(\pi\sb{0},t\sp{\ast}(j\sp{\ast} M)\sb{1})&
~\xrw{\xi'}~\Hom\sb{\clC/Z\sb{0}}(Z\sb{0},j\sp{\ast} M\sb{1})\\
&\xrw{\vartheta'}~\Hom\sb{\clC/Z}(Z,M)~\rw~0.
\end{split}
\end{myeq}
\end{proposition}

\begin{proof}
This follows from \cite[Proposition 3.5]{BP2019} combined with
\cite[Lemma 3.7]{BP2019}, for \w[,]{\xi'=[(tq)\sp{\ast}]\sp{-1}\circ\xi}
$\xi$ given by \cite[Definition 3.4]{BP2019},
and \w{\vartheta'} given by \cite[Definition 3.6]{BP2019}.
\end{proof}

If we further assume that $\clC$ has a faithful ``diagonal nerve'' functor embedding
it in \w{\hy{\SO}{\Cat}} \wwh e.g., for  we obtain from
\cite[Lemma 3.7]{BP2019} and \cite[Lemma 5.2]{BP2019} the following variant of
Proposition \ref{pro3-longseq}:

\begin{proposition}\label{pses}
For $Z$ and $M$ as in Definition \ref{dsplitt},
there is a short exact sequence of abelian groups
\begin{myeq}\label{eqses}
\begin{split}
0~\rw~\pi\sb{1}\map\sb{\Gpd(\clC)/Z}(Z,M)&
~\xrw{\xi''}~\Hom\sb{\clC/Z\sb{0}}(Z\sb{0},j\sp{\ast} M\sb{1})\\
&~\xrw{\vartheta'}~\Hom\sb{\clC/Z}(Z,M)~\rw~0.
\end{split}
\end{myeq}
\end{proposition}

Here \w[,]{\xi'':=\xi'\circ\psi} where
\begin{myeq}\label{eqisom52}
\psi:\pi\sb{1}\map\sb{\Gpd(\clC)/X}(Z,M)\xrw{\cong}
  \Hom\sb{\clC/\pi\sb{0}}(\pi\sb{0},t\sp{\ast}(j\sp{\ast} M)\sb{1})
\end{myeq}
\noindent is the isomorphism of \cite[Lemma 5.2]{BP2019}.

Specializing to the case \w{\clC=\Trk{n-1}} we have our final version:

\begin{proposition}\label{pnewses}
For \w[,]{Y\in\Spl(\Trk{n-1})} \w[,]{Z=H(Y)} and
\w{\widetilde{M}\in[(\Trk{n},Z\sb{0})/Z]\sb{\ab}} a module over $Z$,
there is a short exact sequence of abelian groups
\begin{myeq}\label{eqnewses}
\begin{split}
0~\rw~\pi\sb{1}\map\sb{\Trk{n}/Z}(Z,\widetilde{M})~\xrw{j\sp{\ast}}&~
\pi\sb{1}\map\sb{\Trk{n}/\dZ\sb{0}}(\dZ\sb{0},j\sp{\ast}\widetilde{M})\\
&~\rw~\Hom\sb{\Trk{n}/Z}(Z,\widetilde{M})~\rw~0~.
\end{split}
\end{myeq}
\end{proposition}

\begin{proof}
In order to apply the isomorphism \wref{eqisom52} of \cite[Lemma 5.2]{BP2019}
to the middle term of \wref[,]{eqses} we must rewrite
\w{A:=\Hom\sb{\clC/Z\sb{0}}(Z\sb{0},j\sp{\ast}\widetilde{M}\sb{1})} in the form
of the first term of \wref{eqshortes} (which is the input needed for
\cite[Lemma 5.2]{BP2019}) \wh that is
we want to identify $A$ with an abelian group of the form
\w[,]{\Hom\sb{\clC/\widehat{\pi}\sb{0}}
  (\widehat{\pi}\sb{0},\hat{t}\sp{\ast}(\hat{j}\sp{\ast}\widehat{M})\sb{1})}
so necessarily \w[.]{\widehat{\pi}\sb{0}=Z\sb{0}}
Similarly, we must identify
\w[,]{B:=\pi\sb{1}\map\sb{\Gpd(\clC)/\dZ\sb{0}}(\dZ\sb{0},j\sp{\ast}M)}
the middle term of \wref[,]{eqnewses}
with \w[,]{\pi\sb{1}\map\sb{\Gpd(\clC)/\widehat{Z}}(\widehat{Z},\widehat{M})}
where \w[.]{\widehat{Z}=\dZ\sb{0}} We see that we must choose
\w{\widehat{M}=j\sp{\ast}\widetilde{M}} and
\w[.]{\widehat{Z}:=H(Z\sb{0}\leftrightarrows\widehat{\pi}\sb{0}):=Z\sb{0}}
Thus both \w{\hat{t}:\widehat{\pi}\sb{0}\to Z\sb{0}} and
\w{\hat{j}:\dd\widehat{Z}\sb{0}\to\widehat{Z}=\dZ\sb{0}} are identity maps, and
so \w[,]{\hat{t}\sp{\ast}(\hat{j}\sp{\ast}\widehat{M})\sb{1}=j\sp{\ast}\widetilde{M}}
as required to obtain the middle term of \wref{eqnewses} from \wref[.]{eqses}
\end{proof}

\begin{theorem}\label{the1-longseq}
For any \w{X\in\Trk{n}} and module $M$ over \w[,]{\p{1}X\in\Trk{}}
there is a long exact sequence of cohomology groups
\begin{myeq}\label{eqlongseq}
  \rw \HAQ{s}(\dI X;M)\rw \HAQ{s}(\dI X\sb{0};M)\rw
\HAlg{s}(X;M)\rw \HAQ{s-1}(\dI X;M)\cdots\;.
\end{myeq}
\end{theorem}

\begin{proof}
For each \w{s\geq 0}, apply Proposition \ref{pnewses} to \w{Z=\clK\sb{s} X} and
the module
$$
\widetilde{M}:=E\up{n}(\p{1}\clK\sb{s} X,\delta\sp{\ast}M)\in
[(\Trk{n},(\clK\sb{s} X)\sb{0})/\clK\sb{s} X]\sb{\ab}~,
$$
\noindent where \w{\delta:\clK\sb{s} X\to X} is the unique iterated face map
(followed by the augmentation.
We obtain a short exact sequence of abelian groups
\begin{equation*}
\begin{split}
  0 & \rw \pi\sb{1}\map\sb{\Sz{\Trk{n}}/\clK\sb{s} X}(\clK\sb{s}X,\
\EM{n}{\p{1}\clK\sb{s}X}{\delta\sp{\ast}M}) \rw \\
& \rw \pi\sb{1}\map((\clK\sb{s}X)\sb{0},\EM{n}{\p{1}\clK\sb{s}X}{\delta\sp{\ast}M})
\rw \\
& \to\Hom\sb{\Sz{\Trk{n}}/\clK\sb{s} X}
(\clK\sb{s} X,\EM{n}{\p{1}\clK\sb{s}X}{\delta\sp{\ast}M})\rw 0\;.
\end{split}
\end{equation*}

In turn, this induces a long exact cohomology sequence
\begin{equation*}
\begin{split}
  \cdots\rw & \pi\sp{s}
  (\pi\sb{1}\map\sb{\Sz{\Trk{n}}/\clK\sb{s} X}(\clK\sb{s} X,
  \EM{n}{\p{1}\clK\sb{s}X}{d\sp{\ast}\sb{i}M})\rw \\
\rw  & \pi\sp{s}(\pi\sb{1}\map((\clK\sb{s} X)\sb{0},
\EM{n}{\p{1}\clK\sb{s}X}{\delta\sp{\ast} M)})\rw \\
\rw  & \pi\sp{s}(\Hom\sb{\Sz{\Trk{n}}/\clK\sb{s} X}
(\clK\sb{s} X,\EM{n}{\p{1}\clK\sb{s}X}{\delta\sp{\ast}M})\to\cdots
\end{split}
\end{equation*}

Now by Definition \ref{def1-longseq}
$$
  \pi\sp{s}(\Hom\sb{\Sz{\Trk{n}}/\clK\sb{s} X}
  (\clK\sb{s}X,\EM{n}{\p{1}\clK\sb{s}X}{d\sp{\ast}\sb{i}M})=\HAlg{s}(X;M)~.
$$
Finally, as in \cite[Theorem 4.9 and Corollary 4.10]{BP2019}:
\begin{equation*}
\begin{split}
& \pi\sp{s} (\pi\sb{1}\map\sb{\Sz{\Trk{n}}/\clK\sb{s} X}
  (\clK\sb{s} X,\EM{n}{\p{1}\clK\sb{s}X}{\delta\sp{\ast}M})=
  \HAQ{s}(\dI X;M) \\
& \pi\sp{s}(\pi\sb{1}\map((\clK\sb{s} X)\sb{0},
\EM{n}{\p{1}\clK\sb{s}X}{\delta\sp{\ast}M}) = \HAQ{s}(\dI X\sb{0};M)~.
\end{split}
\end{equation*}
Hence the result follows.
\end{proof}

\begin{corollary}\label{cor1-longseq}
Given \w[,]{X\in\Trk{n}} a module $M$ over \w[,]{\p{1}X} and \w{S(X)} as in
  Lemma \ref{lem2-longseq}, for each \w{s>1} we have an isomorphism
\begin{myeq}\label{eqisocoh}
\HAQ{s+1}(\dI X;M)~\cong~\HAlg{s}(S(X);M)\;.
\end{myeq}
\end{corollary}

\begin{proof}
Since \w{S(X)\sb{0}} is homotopically discrete,
$$
\HAQ{s}(\dI S(X)\sb{0},M)~\cong~\HAQ{s}(\p{0}S(X)\sb{0},M)~.
$$
\noindent By construction of \w[,]{S(X)} \w{\p{0}S(X)\sb{0}} is a free category, so
\w[.]{\HAQ{s}(\p{0}(SX)\sb{0},M)=0}
Thus the long exact sequence of Theorem \ref{the1-longseq} applied to \w{S(X)}
yields \wref[,]{eqisocoh} as required.
\end{proof}

%
%

\end{document}